\documentclass[11pt,a4paper]{article}

\usepackage{graphicx,latexsym,euscript,makeidx,color,bm}
\usepackage{amsmath,amsfonts,amssymb,amsthm,thmtools,mathtools,mathrsfs,enumerate}
\usepackage[colorlinks,linkcolor=blue,citecolor=red]{hyperref}
\usepackage{tikz}

\usepackage{geometry}
\allowdisplaybreaks[4]
   \def\cD{{\cal D}}
\def\dbE{\mathbb{E}}   
\def\dbF{\mathbb{F}} \def\sF{\mathscr{F}}  \def\cF{{\cal F}}
   
\def\dbH{\mathbb{H}}

   \def\cM{{\cal M}}

\def\dbP{\mathbb{P}}   
   
\def\dbR{\mathbb{R}}   
\def\dbS{\mathbb{S}}   
   
   \def\cU{{\cal U}}

\def\ss{\smallskip}             \def\hb{\hbox}
\def\ms{\medskip}              
          \def\lan{\langle}       
\def\h{\widehat}     \def\ran{\rangle}       
\def\q{\quad}               
\def\qq{\qquad}             
         \def\sc{\scriptscriptstyle}
         \def\scT{\sc T}
     \def\ds{\displaystyle}
\def\rf{\eqref}        
\def\cd{\cdot}        \def\({\Big(}           
       \def\){\Big)}           \def\les{\leqslant}
\def\wt{\widetilde}   \def\[{\Big[}           \def\ges{\geqslant}
  \def\]{\Big]}           %\def\tb{\textcolor{blue}}
\def\a{\alpha}       \def\l{\lambda}       
\def\b{\beta}                    
\def\d{\delta}              
\def\e{\varepsilon}      \def\L{\varLambda}
\def\f{\varphi}           
       \def\si{\sigma}    \def\Si{\varSigma}
\def\i{\infty}       \def\z{\zeta}      \def\Th{\varTheta}
\def\be{\begin{equation}}
\def\bel{\begin{equation}\label}
\def\ee{\end{equation}}
\def\bea{\begin{eqnarray}}
\def\eea{\end{eqnarray}}
\def\bt{\begin{theorem}}
\def\et{\end{theorem}}
\def\bc{\begin{corollary}}
\def\ec{\end{corollary}}
\def\bl{\begin{lemma}}
\def\el{\end{lemma}}
\def\bp{\begin{proposition}}
\def\ep{\end{proposition}}
\def\br{\begin{remark}}
\def\er{\end{remark}}
\def\ba{\begin{array}}
\def\ea{\end{array}}
\def\bd{\begin{definition}}
\def\ed{\end{definition}}
\def\5n{\negthinspace \negthinspace \negthinspace \negthinspace \negthinspace }
\def\4n{\negthinspace \negthinspace \negthinspace \negthinspace }
\def\3n{\negthinspace \negthinspace \negthinspace }
\def\2n{\negthinspace \negthinspace }
\def\1n{\negthinspace }

\newtheorem{theorem}{Theorem}[section]
\newtheorem{proposition}[theorem]{Proposition}
\newtheorem{corollary}[theorem]{Corollary}
\newtheorem{lemma}[theorem]{Lemma}

\theoremstyle{definition}
\newtheorem{definition}[theorem]{Definition}
\newtheorem{example}[theorem]{Example}

\newtheorem{assumption}[theorem]{Assumption}

\theoremstyle{remark}
\newtheorem{remark}[theorem]{Remark}

\def\punct{}
\newtheoremstyle{dotless}{}{}{\rm}{}{\bf}{\punct}{.5em}{}
\theoremstyle{dotless}

\makeatletter
   
   \@addtoreset{equation}{section}
   \newcommand{\setword}[2]{%
   \phantomsection
   #1\def\@currentlabel{\unexpanded{#1}}\label{#2}%
   }
\makeatother

\def\3n{\negthinspace \negthinspace \negthinspace }
\def\2n{\negthinspace \negthinspace }
\def\1n{\negthinspace }
\def\bel{\begin{equation}\label}

\def\dbE{\mathbb{E}}
\def\dbF{\mathbb{F}}

\def\dbH{\mathbb{H}}

\def\dbP{\mathbb{P}}

\def\dbR{\mathbb{R}}
\def\dbS{\mathbb{S}}

\def\sF{\mathscr{F}}

\def\ds{\displaystyle}

\def\ns{\noalign{\ss}}
\def\a{\alpha}
\def\b{\beta }

\def\d{\delta}
\def\e{\varepsilon}
\def\z{\zeta}

\def\l{\lambda}

\def\si{\sigma}

\def\f{\varphi}

\def\i{\infty}
\def\Th{\Theta}
\def\L{\Lambda}
\def\Si{\Sigma}

\def\O{\Omega}
\def\cD{{\cal D}}

\def\cF{{\cal F}}

\def\cM{{\cal M}}

\def\cU{{\cal U}}

\def\BS{{\bf S}}

\def\Th{\Theta}
\def\L{\Lambda}
\def\Si{\Sigma}

\def\O{\Omega}

\def\ss{\smallskip}
\def\ms{\medskip}

\def\q{\quad}
\def\qq{\qquad}
\def\hb{\hbox}

\def\h1{\outline{$1$}}

\def\hh1{\outline{$1$}}
\def\hh2{\outline{$2$}}
\def\hh3{\outline{$3$}}
\def\hh4{\outline{$4$}}
\def\hh5{\outline{$5$}}
\def\hh6{\outline{$6$}}
\def\hh7{\outline{$7$}}
\def\hh8{\outline{$8$}}
\def\hh9{\outline{$9$}}
\def\hh0{\outline{$0$}}

\def\limsup{\mathop{\overline{\rm lim}}}
\def\liminf{\mathop{\underline{\rm lim}}}

\def\lan{{\langle}}
\def\ran{{\rangle}}

\def\h{\widehat}
\def\wt{\widetilde}

\def\cd{\cdot}
\def\cds{\cdots}

\def\les{\leqslant}
\def\ges{\geqslant}

\def\({\Big (}
\def\){\Big )}
\def\[{\Big[}
\def\]{\Big]}
\def\lan{\langle}
\def\ran{\rangle}

\def\bde{\begin{definition}\label}
	\def\ede{\end{definition}}
\def\bel{\begin{equation}\label}
		\def\ee{\end{equation}}
	\def\bt{\begin{theorem}\label}
		\def\et{\end{theorem}}
	\def\bc{\begin{corollary}\label}
		\def\ec{\end{corollary}}
	\def\bl{\begin{lemma}\label}
		\def\el{\end{lemma}}
	\def\bp{\begin{proposition}\label}
		\def\ep{\end{proposition}}
	\def\bex{\begin{example}\label}
		\def\ex{\end{example}}
	\def\bas{\begin{assumption}}
		\def\eas{\end{assumption}}
	\def\br{\begin{remark}\label}
		\def\er{\end{remark}}
	\def\ba{\begin{array}}
		\def\ea{\end{array}}
	\def\ed{\end{document}}

\def\rf{\eqref}

\def\square#1{\vbox{\hrule\hbox{\vrule height#1%
			\kern#1\vrule}\hrule}}
\def\rectangle#1#2{\vbox{\hrule\hbox{\vrule height#1%
			\kern#2\vrule}\hrule}}

\font\tenbb=msbm10 \font\sevenbb=msbm7 \font\fivebb=msbm5

\newfam\bbfam
\scriptscriptfont\bbfam=\fivebb \textfont\bbfam=\tenbb
\scriptfont\bbfam=\sevenbb

\makeatletter

\@addtoreset{equation}{section}
\makeatother

\usepackage{xcolor}
\makeatletter

\input pdf-trans
\newbox\qbox
\def\usecolor#1{\csname\string\color@#1\endcsname\space}
\newcommand\bordercolor[1]{\colsplit{1}{#1}}
\newcommand\fillcolor[1]{\colsplit{0}{#1}}
\newcommand\outline[1]{\leavevmode%
	\def\maltext{#1}%
	\setbox\qbox=\hbox{\maltext}%
	\boxgs{Q q 2 Tr \thickness\space w \fillcol\space \bordercol\space}{}%
	\copy\qbox%
}
\makeatother
\newcommand\colsplit[2]{\colorlet{tmpcolor}{#2}\edef\tmp{\usecolor{tmpcolor}}%
	\def\tmpB{}\expandafter\colsplithelp\tmp\relax%
	\ifnum0=#1\relax\edef\fillcol{\tmpB}\else\edef\bordercol{\tmpC}\fi}
\def\colsplithelp#1#2 #3\relax{%
	\edef\tmpB{\tmpB#1#2 }%
	\ifnum `#1>`9\relax\def\tmpC{#3}\else\colsplithelp#3\relax\fi
}
\bordercolor{black}
\fillcolor{white}
\def\thickness{.3}

\def\h1{\outline{$1$}}

\begin{document}

\title{\bf Turnpike Property of  Mean-Field Linear-Quadratic Optimal Control Problems in Infinite-Horizon with Regime Switching}
		
		\author{  Hongwei Mei\footnote{ Department of Mathematics and Statistics, Texas Tech University, Lubbock, TX 79409, USA; email: {\tt hongwei.mei@ttu.edu}. This author is partially supported by Simons Travel Grant MP-TSM-00002835.},~~ Svetlozar Rachev\footnote{ Department of Mathematics and Statistics, Texas Tech University, Lubbock, TX 79409, USA; email: {\tt zari.rachev@ttu.edu}}~~ and ~ Rui Wang\footnote{ Department of Mathematics and Statistics, Texas Tech University, Lubbock, TX 79409, USA; email: {\tt rui-math.wang@ttu.edu}}}
			
\maketitle

\paragraph{Abstract.} This paper considers an optimal control problem for a linear mean-field stochastic differential equation having regime switching with quadratic functional in the large time horizons. Our main contribution lies in  establishing the strong turnpike property for the optimal pairs when the time horizon tends to infinity. To work with the mean-field terms,  we apply the orthogonal decomposition method to derive a closed-loop representation of the optimal control problem in a finite time horizon. To analyze the asymptotic behavior of the optimal controls, we examine the convergence of the solutions of  Riccati equations and backward differential equations as the time horizon tends to infinity. The strong turnpike property can be obtained based on these convergence results. Finally, we verify the optimality of the limit optimal pair in two cases: integrable case and local-integrable case.

\ms

\ms
\newpage

\section{Introduction}

Let $(\Omega,\sF,\dbP)$ be a complete probability space on which a standard one-dimensional Brownian motion $W=\{W(t);\,t\ges 0\}$ and a  Markov chain $\a(\cd)$ with a finite state space $\cM=\{1,2,3,\cds,m_0\}$ are defined, for which they are assumed to be independent. The generator of $\a(\cd)$ is denoted by $(\l_{\imath\jmath})_{m_0\times m_0}$. We now denote by $\dbF^W=\{\sF^W_{t}\}_{t\ges0}$ (resp. $\dbF^\a=\{\sF_{t}^\a\}_{t\ges0}$, $\dbF=\{\sF_{t}\}_{t\ges0})$ the usual augmentation of the natural filtration generated by $W(\cd)$ (resp. by $\a(\cd)$, and by $(W(\cd),\a(\cd))$). Write
\begin{align*}
& \dbE_t^\a[~\cd~]=\dbE[~\cd~|\cF_t^\a].\end{align*}

Consider the following controlled mean-field linear stochastic differential equation (MF-SDE, for short), with regime switchings:
\begin{align}\label{state} \begin{cases}&dX(t)=\Big[A(\a(t))X(t)+\bar A(\a(t))\dbE_t^\a[X(t)]\\
&\qq\qq\qq\qq+B(\a(t))u(t)+\bar B(\a(t))\dbE_t^\a[u(t)]+b(t)\Big]dt\\
&\q+\Big[C(\a(t))X(t)+\bar C(\a(t))\dbE_t^\a[X(t)]\\
&\qq\qq\qq\qq+D(\a(t))u(t)+\bar B(\a(t))\dbE_t^\a[u(t)]+\si(t)\Big]dW(t),\q t\geq s\\
& X(s)=x,\q\a(s)=\imath,
\end{cases}
\end{align}
under  the following quadratic cost functional
\begin{align}\label{cost}
& J_{\scT}(s,x,\imath;u(\cd)) =\dbE \int_s^{T}
\Big[\begin{pmatrix}X(t)\\u(t)\end{pmatrix}^\top\begin{pmatrix}Q(\a(t))& S^\top(\a(t))\\
S (\a(t))&R(\a(t))\end{pmatrix}
\begin{pmatrix}X(t)\\u(t)\end{pmatrix}\nonumber
\\
&\qq\qq
+\begin{pmatrix}\dbE_t^\a[X(t)]\\\dbE_t^\a[u(t)]\end{pmatrix}^\top\begin{pmatrix}\bar Q(\a(t))& \bar S^\top(\a(t))\\
\bar S (\a(t))&\bar R(\a(t))
\end{pmatrix}\begin{pmatrix}\dbE_t^\a[X(t)]\\\dbE_t^\a[u(t)]\end{pmatrix}
  \nonumber\\
 &\qq\qq +\lan q(t), X(t)\ran+\lan \bar q(t), \dbE_t^\a[X(t)]\ran+\lan r(t), u(t)\ran+\lan \bar r(t), \dbE_t^\a[u(t)]\ran\Big]dt.
 \end{align}

We assume the following throughout the paper.\ms

{\noindent \bf (A1)}.  (1) $A(\cd),\bar A(\cd), C(\cd),\bar C(\cd) :\cM\to\dbR^{n\times n}$; $B(\cd),\bar B(\cd),D(\cd),\bar D(\cd):\cM\to\dbR^{m\times n}$.

(2) $Q(\cd), \bar Q(\cd):\cM\mapsto \dbS^n$; $R(\cd),\bar R(\cd):\cM\to \dbS^n$; $S(\cd),\bar S(\cd):\cM\to \dbR^{n\times m}$.

(3) $q(\cd),\bar q(\cd)\in L_{\dbF}^2(0,T;\dbR^n)$, $r(\cd),\bar r(\cd)\in L_{\dbF}^2(0,T;\dbR^n)$ for any $T>0.$

\ms

Here, the superscript $\top$ denotes the transpose of matrices; $\lan\cd\,,\cd\ran$ denotes the inner product of two vectors (possibly in different spaces). The $\dbS^n$, $\dbS^n_+$ and $\dbS^n_{++}$ are defined by  the sets of all $(n\times n)$ symmetric, positive semi-definite, and positive definite matrices, respectively. For any Euclidean space $\dbH$ (such as $\dbR^n,\dbR^{n\times m}$, etc.),
\begin{align*}& L^2_\dbF(s,T;\dbH):=\Big\{\f:[s,T]\times\O\to\dbH\bigm|\f(\cd)\hb{ is $\dbF$-progressively measurable}\\
&\qq\qq\qq\qq\text{ with }\dbE\int_s^T|\f(t)|_\dbH^2dt
<\i\Big\}.\end{align*}
%In the above, the set $[t,T]$ is understood as the following:
%$$[t,T)]=\left\{\2n\ba{ll}
%\ns\ds[t,T],\qq~ T<\i,\\
%\ns\ds[t,\i),\qq T=\i.\ea\right.$$
%
We also write $L_{\cF_{t}}^2(\dbH)$ by the set of $\cF_t$-measurable, $\dbH$-valued random variables with finite second moment.

 In \rf{state}, any $(s,x,\imath)$ is called an {\it initial pair} if $(x,\imath)\in  L_{\cF_{s}}^2(\dbR^n)\times\cM$ and $s\in[0,\i)$. Write the set of all initial pairs by $\cD$. When $T<\infty$, 
the  {\it control} process $u(\cd)$ is taken from the space
$$\cU[s,T]=L_\dbF^2(s,T;\dbR^m).$$
Provided (A1), it is well-known that for each $(s,x,\imath)\in\cD$ and $u(\cd)\in\cU[s,T]$, \rf{state} admits a unique solution $X(\cd)\equiv X(\cd\,;s,x,\imath;u(\cd))\in L_{\dbF}^2(s,T;\dbR^n)$. Consequently, the cost functional  $J_{\scT}(s,x,\imath;u(\cd))$ is finite for all $u(\cd)\in\cU[s,T]$. Then it is natural to consider the following optimal control problem.

\ms

{\bf Problem (MF-LQ)$_{\scT}$.} For a given initial pair $(s,x,\imath)\in\cD$, find a control $\bar u_{\scT}^{s,x,\imath}(\cd)\in\cU[s,T]$ such that
\bel{opt-u3}J_{\scT}(s,x,\imath;\bar u_{\scT}^{s,x,\imath}(\cd))=\inf_{u(\cd)\in\cU[s,T]}J_{\scT}(s,x,\imath;u(\cd))\equiv V_{\scT}(s,x,\imath).\ee

 Problem (MF-LQ)$_{\scT}$ is usually referred to as   {\it mean-field linear-quadratic} (MF-LQ, for short) {\it optimal control problems} with regime switchings over a finite horizon.   Under some mild conditions,  Problem (MF-LQ)$_{\scT}$  admits a unique (open-loop) optimal control $\bar u_{\scT}^{s,x,\imath}(\cd)\in \cU[s,T]$. Write $\bar X_{\scT}^{s,x,\imath}(\cd)$ by the corresponding {\it  optimal state processes}.

For the cases without mean-field terms, it is proven in \cite{Mei-Wang-Yong-2025a,Mei-Wang-Yong-2025b} that  there exists some stochastic processes $(\bar X_{\sc\i}(\cd),\bar u_{\sc\i}(\cd))$ with initial $(0,x,\imath)$, some absolute constants $\b,K>0$, and a function $h(\cd):[0,\i)\rightarrow [0,\infty)$ independent of $0<T<\i$, such that
\bel{STP0}\dbE\(|\bar X_{\scT}^{0,x,\imath }(t)-\bar X_{\sc\i}(t)|^2+\int_0^t|\bar u_{\scT}^{0,x,\imath}(r)-\bar u_{\sc\i}(r)|^2dr\)\les Ke^{-\b(T-t)}\(e^{-\b t}|x|^2+h(t)\),\ee
for all  $t\in[0,T]$.  Such an asymptotic behavior is called the {\it strong turnpike property} (STP, for short) for the optimal pair $(\bar X_{\scT}^{0,x,\imath}(\cd),\bar u_{\scT}^{0,x,\imath}(\cd))$  as $T\to\i$.

Investigation on turnpike property (for deterministic economics systems) can be dated back to von Neumann \cite{Ramsey-1928, Neumann-1945} where the name {\it turnpike property} intuitively was suggested by the highway system of the United States in \cite{Dorfman-Samuelson-Solow-1958}. Since then the turnpike property has been found to hold for a large class of (deterministic, finite or infinite dimensional) optimal control problems. Numerous relevant results can be found in \cite{McKenzie-1976,Carlson-Haurie-Leizarowitz-1991,Damm-Grune-Stieler-Worthmann-2014,
Trelat-Zuazua-2015,Zuazua-2017,Grune-Guglielmi-2018,Zaslavski-2019,
Lou-Wang-2019,Breiten-Pfeiffer-2020,Sakamoto-Zuazua-2021,Faulwasser-Grune-2022} and the references cited therein.  In particular,  \cite{Carlson-Haurie-Leizarowitz-1991} deals with some stochastic systems with jumps and the corresponding turnpike property was studied. At about the same time, certain stability for a finite time horizon multi-person discrete stochastic game was investigated and using the idea of turnpike property, it was shown the existence of an equilibrium for a stationary (discrete random) games (\cite{Marimon-1989}). For  continuous-time stochastic optimal LQ control problems, one can refer to  \cite{Sun-Wang-Yong-2022,Conforti-2023,Chen-Luo-2023,Sun-Yong-2024,Sun-Yong-2024b,
Jian-Jin-Song-Yong-2024,Schiessl-Baumann-Faulwasser-Grune-2024,
Bayraktar-Jian-2025}. For stochastic LQ systems with Markovian jumps, \cite{Mei-Wang-Yong-2025a, Mei-Wang-Yong-2025b}  generalizes the Turnpike property to three different cases: homogeneous cases, integrable cases and non-integrable cases. In the homogeneous cases, the main effort is devoted to proving the exponential convergence of Riccati equation, while in the latter two cases, the main focus is placed on the convergence of backward stochastic differential equations (BSDEs). It is worth mentioning that the non-homogeneous coefficients in \cite{Mei-Wang-Yong-2025b} are allowed to be a  stochastic process instead of deterministic constants which are necessary in the previous literature such as \cite{Sun-Yong-2024,Sun-Yong-2024b}.  Recently, \cite{Li-Wu-Zhang-2025} establishes the similar results for homogeneous two player zero-sum games.  

In this paper, we focus on the mean-field stochastic optimal control with switching \eqref{state} under the cost functional \eqref{cost}, which generals the LQ optimal control problem studied in  \cite{Mei-Wang-Yong-2025a, Mei-Wang-Yong-2025b}. In particular,  the  mean-field interactions $\dbE^\a_t [X(t)]$ and $\dbE^\a_t[u(t)]$ are involved. Similarly, we will prove two types of asymptotic behaviors of the open-loop optimal pair to Problem (MF-LQ)$_{\scT}$:

\ms

$\bullet$ {\bf Integrable Case}: $b(\cd),\si(\cd),q(\cd),\bar q(\cd)\in L_\dbF^2(0,\i;\dbR^n)$ and $r(\cd),\bar r(\cd)\in L_\dbF^2(0,\i;\dbR^m)$. In this case, $h(\cd)$ is a non-negative integrable function on $[0,\i)$. In particular, when $b(\cd),\si(\cd),q(\cd),r(\cd), \bar q(\cd),\bar r(\cd)$ are all 0 (i.e. homogeneous case), we have  $h(t)\equiv0$. In the integrable case, we will see that the  $(\bar X_{\sc\i}(\cd),\bar u_{\sc\i}(\cd))$ is the optimal control for an infinite-horizon problem.

\ms

$\bullet$ {\bf Local-Integrable Case}: For any $0<T<\i$, $b(\cd),\si(\cd),q(\cd),\bar q(\cd)\in L_\dbF^2(0,T;\dbR^n)$ and $r(\cd),\bar r(\cd)\in L_\dbF^2(0,T;\dbR^m)$ with some additional assumptions. In this case, we can take $h(t)\equiv1$. In the local-integrable case, we will see that the  $(\bar X_{\sc\i}(\cd),\bar u_{\sc\i}(\cd))$ is the optimal control for an ergodic control problem.
\ms

Based on  \cite{Mei-Wang-Yong-2025a, Mei-Wang-Yong-2025b}, the main novelty of this paper lies in studying the convergence of a system of Riccati equations and mean-field BSDEs. The rest of the paper is arranged as follows. In Section \ref{sec:opt}, we  obtain an equivalent formulation of Problem (MF-LQ)$_{\scT}$ through  orthogonal decomposition method to  derive the closed-loop representation of the optimal control. Then Section \ref{sec:abo} studies the asymptotic behavior of the optimal controls as $T\rightarrow\infty$ where the main efforts are placed on the convergence of Riccati equations and mean-field BSDEs. Our main results on STP are proved in Section \ref{sec:stp}, together with  the corresponding optimality for integrable case and local-integrable case. Finally, some concluding remarks are made in Section \ref{sec:con}.

\section{Optimal Control for Problem (MF-LQ)$_{\scT}$}\label{sec:opt}

In this section, we will recall some results in \cite{Mei-Wei-Yong-2025} on the optimal control $(\bar X_{\scT}^{s,x,\imath}(\cd),\bar u_{\scT}^{s,x,\imath}(\cd))$ for Problem (MF-LQ)$_{\scT}$. The section is divided into several subsections.

\subsection{Martingale Measure}
Recall that $\a(\cd)$ is a Markov chain whose state space $\cM$ is finite. Thus, we may let its generator be $(\l_{\imath\jmath})_{m_0\times m_0}\in\dbR^{m_0\times m_0}$, which is a real matrix so that the following hold:
\bel{q-prop}\l_{\imath\jmath}>0,\q\imath\ne\jmath;\qq\sum_{\jmath=1}^{m_0}
\l_{\imath\jmath}=0,\q\imath\in\cM.\ee
We now proceed with a {\it martingale measure of Markov chain} $\a(\cd)$. For $\imath\ne\jmath$, we define
$$\ba{ll}
\ns\ds\wt M_{\imath\jmath}(t):=\sum_{0\leq s\les t}{\bf1}_{[\a(s_-)=\imath]}{\bf1}_{[\a(s)=\jmath]}\equiv\hb{accumulative jump number from $\imath$ to $\jmath$ %\ne\imath$
in $(0,t]$},\\
\ns\ds\lan\wt M_{\imath\jmath}\ran(t):=\int_0^t
\l_{\imath\jmath}{\bf1}_{[\a(s-)=\imath]}ds,%\\
%\ns\ds
\q M_{\imath\jmath}(t):=\wt M_{\imath\jmath}(t)-\lan\wt M_{\imath\jmath}\ran(t),\qq s\ges0.\ea$$
The above $M_{\imath\jmath}(\cd)$ is a square-integrable martingale (with respect to $\dbF^\a$). For convenience, we let
$$M_{\imath\imath}(t)=\wt M_{\imath\imath}(t)=\lan\wt M_{\imath\imath}\ran(t)=0,\qq s\ges0.$$
Then $\{M_{\imath\jmath}(\cd)\bigm|\imath,\jmath\in\cM\}$ is the {\it martingale measure} of Markov chain $\a(\cd)$. 

Now, let $\dbF_-$ be the smallest filtration containing $\{\cF_t^W\}_{t\ges 0}$ and $\{\cF_{t-}^\a\}_{t\ges 0}$ augumented with all $\dbP$-null sets.
To define the stochastic integral with respect to such a martingale measure, we need to introduce the following Hilbert spaces
$$\ba{ll}
\ns\ds M^2_{\dbF_-}(t,T;\dbH)=\Big\{\f(\cd\,,\cd)=(\f(\cd\,,1),\cds,\f(\cd\,,m_0))\bigm|
\f(\cd\,,\cd)\hb{ is $\dbH$-valued and $\dbF_-$-measurable }\\
\ns\ds\qq\qq\qq\qq\qq \hb{ with } \dbE\int_t^T\sum_{\imath\neq\jmath}|
\f(s,\jmath)|^2\l_{\imath\jmath}{\bf1}_{[\a(s^-)=\imath]}d\widehat M_{ij}(s)<%\dbE\int_s^T\L[|\f(r,\cd)|^2](\a(s))ds
\i,\q\forall\imath,\jmath\in\cM
\Big\}.\ea$$
Now, for any $\f(\cd)\in M^2_{\dbF_-}(t,T;\dbH)$, we define its stochastic integral against $dM$ by the following:
$$\int_t^T\f(s)dM(s):=\sum_{\jmath\ne\imath}\int_{[t,T]}\f(r,\jmath)
{\bf1}_{[\a(s^-)=\imath]}dM_{\imath\jmath}(s),$$
whose quadratic variation is
$$\dbE\(\int_t^T\f(s)dM(s)\)^2=\dbE\int_t^T\sum_{\imath\neq\jmath}|
\f(s,\jmath)|^2\l_{\imath\jmath}{\bf1}_{[\a(s)=\imath]}ds.%\dbE\int_s^T\L[|\f(r,\cd)|^2](\a(s))ds,
$$

\subsection{Orthgonal Decomposition}
In  this section, we will derive an equivalent formulation for Problem (MF-LQ)$_{\scT}$. In addition, we will also propose two optimal control problems over the infinite horizon, which will be used in verifying the optimality for the limit process in the later section.

 For any $\varphi(\cd)\in L^2_{\dbF}(s,T;\dbH)$, define 
$$\Pi[\varphi](t)=\dbE^\a_t[\varphi(t)],\text{ for each $t\in(s,T].$}$$
Note that $\Pi[\varphi](t)\in\cF_t^\a$ and the definition is in point-wise sense. For any $\varphi_{\sc 1}(\cd)=\varphi_2(\cd)\in L^2_{\dbF}(s,T;\dbH)$,  it follows that 
$$\dbE\int_s^T\big|\Pi[\varphi_{\sc 1}](t)-\Pi[\varphi_{\sc 2}](t)\big|^2dt\les\dbE\int_s^T
		\big|\varphi_{\sc 1}(t) -\varphi_{\sc 2}(t)\big|^2dt=0.$$
This yields that $\Pi$ defines a linear map from 
$L^2_{\dbF}(s,T;\dbH)$ to $L^2_{\dbF^\a}(s,T;\dbH)$. Note that  for any $\varphi(\cd)\in L^2_{\dbF}(s,T;\dbH)$,
$$\int_s^T\lan \Pi[\varphi](t),\varphi(t)-\Pi[\varphi](t)\ran dt=0.$$
Therefore $\Pi$ induces the following orthogonal decomposition 
\begin{align*}L^2_{\dbF}(s,T;\dbH)= L^2_{\dbF^\a}(s,T;\dbH)^\perp\oplus L^2_{\dbF^\a}(s,T;\dbH)
		\end{align*}
It can be easily seen that the above also holds for $T=\i.$ With such a decomposition, we will reformulate Problem (MF-LQ)$_{\scT}$ in the product space instead.

Now we apply the orthogonal decomposition on Problem (MF-LQ)$_{\scT}$. Write 
\begin{align*}\begin{cases}&X_{\sc 1}(t)=X(t)-\dbE_t^\a[X(t)],\q X_{\sc 2}(t)=\dbE_t^\a[X(t)],\\
&u_{\sc 1}(t)=u(t)-\dbE_t^\a[u(t)],\q u_{\sc 2}(t)=\dbE_t^\a[u(t)],\\
&x_{\sc 1}=x-\dbE_s^\a[x],\q x_{\sc 2}=\dbE_s^\a[x].\end{cases}
\end{align*}
 By Lemma A.1 in \cite{Mei-Wei-Yong-2024}, we have 
\begin{align}\label{SDE1ho}\begin{cases}
& dX_{\sc 1}(t)=[A_{\sc 1}(\a(t)) X_{\sc 1}(t)+B_{\sc 1}(\a(t))u_{\sc 1}(t)+b_1(t)]dt\\
&	 \q+[C_{\sc 1}(\a(t)) X_{\sc 1}(t)+C_{\sc 2}(\a(t)) X_{\sc 2}+D_{\sc 1}(\a(t))u_{\sc 1}(t)+D_{\sc 2}(\a(t))u_{\sc 2}(t)+\si(t)]dW(t),\\
& dX_{\sc 2}(t)=[A_{\sc 2}(\a(t))X_{\sc 2}(t)+B_{\sc 2}(\a(t))u_{\sc 2}(t)+b_2(t)]dt, \q t\in[s,T],\\
& X_{\sc 1}(s)=x_{\sc 1},\q X_{\sc 2}(s)=x_{\sc 2}, \q \a(s)=\imath.\end{cases}\end{align}
The cost functional \rf{cost} can be written as
\begin{align}\label{cost1}
&J_{\scT}(s, x_{\sc 1}\oplus x_{\sc 2},\imath;u_{\sc 1}(\cd)\oplus u_{\sc 2}(\cd)):=J_{\scT}(s, x,\imath,;u(\cd))\nonumber\\
&\q=\sum_{k=1}^2\dbE  \int_s^T\[\lan  Q_k(\a(t))X_k(t),X_k(t)\ran +2\lan S_k(\a(t))X_k(t),u_k(t)\ran+\lan R_k(\a(t))u_k(t),u_k(t)\ran\nonumber\\
&\qq\qq\qq+\lan q_k(t), X_k(t)\ran+\lan r_k(t), u_k(t)\ran\]dt.
\end{align}
Here
$
\Gamma_{\sc 1}(\imath)=\Gamma(\imath),\q \Gamma_{\sc 2}(\imath)=\Gamma(\imath)+\bar \Gamma(\imath),\text{ for }\Gamma=A,B,C,D,Q,R,S,q,r.
$

Using such a decomposition, we also rewrite the set of admissible initial states and the set of admissible controls by 
\begin{align*}
&\cD=\Big\{(s,\imath,x_{\sc 1}\oplus x_{\sc 2})\bigm|s\in[0,\i),\imath\in\cM,x_{\sc 1}\in L^2_{\cF^\a_s}(\O;\dbR^n)^\perp, x_{\sc 2}\in L^2_{\cF^\a_s}(\O;\dbR^n)\Big\}.\\
&\cU[s,T]= L_{\dbF^\a}^2(s,T;\dbR^m)^\perp\oplus L_{\dbF^\a}^2(s,T;\dbR^m).
\end{align*}

After the orthogonal decomposition, Problem (MF-LQ)$_{\scT}$ can be equivalently stated as follows.\ms

\noindent {\bf Problem (MF-LQ)$^*_{\scT}$}.  For any $(s,\imath,x_{\sc 1}\oplus x_{\sc 2})\in\cD$, find a $\bar u_{\sc 1}(\cd)\oplus\bar u_{\sc 2}(\cd)\in\cU[s,T]$ such that
$$J_{\scT}(s,x_{\sc 1}\oplus x_{\sc 2},\imath; \bar u_{\sc 1}(\cdot)\oplus \bar u_{\sc 2}(\cdot))=\inf_{u_{\sc 1}(\cdot)\oplus u_{\sc 2}(\cdot)\in \cU[s,T]}J_{\scT}(s,x_{\sc 1}\oplus x_{\sc 2},\imath; u_{\sc 1}(\cdot)\oplus u_{\sc 2}(\cdot)).$$

Our main effort in the sequel is devoted to studying the strong Turnpike property for   the optimal couple for Problem (MF-LQ)$^*_{\scT}$ as $T\rightarrow\i$. To identify the optimality of the limit process, it is natural to arise two infinite-horizon optimal control problems where  a stabilizibility condition is necessary.

\subsection{Stabilizability and Infinite-Horizon Optimal Control Problems}

In this subsection, we propose the following optimal control problems over the infinite horizon $[0,\i)$ to identify the optimality of the limit pair. The following are the two problems.\ms 

\noindent {\bf Problem (MF-LQ)$^*_{\sc\i}$}.  For any $(s,\imath,x_{\sc 1}\oplus x_{\sc 2})\in\cD$, find a $\bar u_{\sc 1}(\cd)\oplus\bar u_{\sc 2}(\cd)\in\cU^{s,x,\imath}_{ad}[s,\i)$ such that
$$J_{\sc\i}(s,x_{\sc 1}\oplus x_{\sc 2},\imath; \bar u_{\sc 1}(\cdot)\oplus \bar u_{\sc 2}(\cdot))=\inf_{u_{\sc 1}(\cdot)\oplus u_{\sc 2}(\cdot)\in\cU^{s,x,\imath}_{ad}[s,\i)}J_{\sc\i}(s,x_{\sc 1}\oplus x_{\sc 2},\imath; u_{\sc 1}(\cdot)\oplus u_{\sc 2}(\cdot)).$$

\noindent {\bf Problem (MF-LQ)$^*_{\sc E}$}.  For any $(s,\imath,x_{\sc 1}\oplus x_{\sc 2})\in\cD$, find a $\bar u_{\sc 1}(\cd)\oplus\bar u_{\sc 2}(\cd)\in\cU^{s,x,\imath}_{ad}[s,\i)$ such that
$$J_{\sc E}(s,x_{\sc 1}\oplus x_{\sc 2},\imath; \bar u_{\sc 1}(\cdot)\oplus \bar u_{\sc 2}(\cdot))=\inf_{u_{\sc 1}(\cdot)\oplus u_{\sc 2}(\cdot)\in\cU_{loc}[s,\i)}J_{\sc E}(s,x_{\sc 1}\oplus x_{\sc 2},\imath; u_{\sc 1}(\cdot)\oplus u_{\sc 2}(\cdot)).$$
Here
$$J_{\sc E}(s,x_{\sc 1}\oplus x_{\sc 2},\imath; u_{\sc 1}(\cdot)\oplus u_{\sc 2}(\cdot))=\liminf_{T\rightarrow\infty}\frac 1 TJ_{\sc T}(s,x_{\sc 1}\oplus x_{\sc 2},\imath; u_{\sc 1}(\cdot)\oplus u_{\sc 2}(\cdot)).$$

In the above two problems, we define
\begin{align*}&\cU_{ad}^{s,x,\imath}[0,\i)=\Big\{ u_{\sc 1}(\cdot)\oplus u_{\sc 2}(\cdot)\in L_{\dbF}^2(s,\i;\dbR^m)\Big|  X_{\sc1}(\cd;x,\imath,u(\cd))\oplus X_{\sc2}(\cd;x,\imath,u(\cd))\in L_{\dbF}^2(s,\i;\dbR^n)
\Big\},\\
&\cU_{loc}[s,\i)=\bigcap_{T>s} \cU[s,T]\end{align*}
where $ X(\cd;x,i,u(\cd))$ is the solution of \eqref{state} with initial $(x,\imath)=(x_1\oplus x_2,\imath)$ and control $u(\cd)=u_1(\cd)\oplus u_2(\cd).$ The admissible control $\cU_{ad}^{s,x,\imath}[0,\i)$ in  Problem (MF-LQ)$^*_{\sc\i}$ is a subset of $\cU[0,\i)$ which is used to guaranttee $J_{\sc\i}(s,x_{\sc 1}\oplus x_{\sc 2},\imath; u_{\sc 1}(\cdot)\oplus u_{\sc 2}(\cdot))$ to be finite (so that  Problem (MF-LQ)$^*_{\sc\i}$ is well-defined).

Problem (MF-LQ)$^*_{\sc \i}$ is usually referred to as the {\it infinite-horizon control problem} and Problem (MF-LQ)$^*_{\sc E}$ is usually referred to as the {\it ergodic control problem}. We will see that if we impose different assumptions on the non-homogeneous terms, then the limit process $(X_{\sc\i}(\cd),u_{\sc\i}(\cd))$ turns out to be the optimal couple for either Problem (MF-LQ)$^*_{\sc \i}$  or Problem (MF-LQ)$^*_{\sc E}$.

We notice that in  Problem (MF-LQ)$^*_{\sc\i}$, the set of admissible controls, $\cU^{s,x,\imath}_{ad}[s,\i)$, is dependent on the initial value $(s,x,\imath)$. Moreover, we also see that  $\cU^{s,x,\imath}_{ad}[s,\i)$ may not be a linear space necessarily. The following is a counter example.
\begin{example}
Consider the following 1-dimensional ordinary differential equation
$$dX(t)=(X(t)+u(t))dt,\q X(0)=x_0.$$
Let $u(t)=-2x_0e^{-t}$. It can be easily seen that $X(t)=x_0e^{-t}$ with $\int_0^\infty|X(t)|^2dt<\infty.$ Therefore  $u(\cd)\in \cU_{ad}^{x_0}[0,\i).$

Let $v(t)=-2\lambda x_0e^{-t}=\lambda u(t)$. Under such a control, the state process satisfies
$$X(t)= e^{t}x_0(1-\lambda)+x_0e^{-t}$$
For any $\lambda\neq 1,$  $v(\cd)=\lambda u(\cd)\notin \cU_{ad}^{x_0}[0,\i) $. Such an example justifies that $\cU^{s,x,\imath}_{ad}[s,\i)$ may not be a linear space necessarily.
\end{example}

Because $\cU^{s,x,\imath}_{ad}[s,\i)$ may not be a linear space, $(u_{\sc 1}(\cd)+\e v_1(\cd))\oplus (u_{\sc 2}(\cd)+\e v_2(\cd))$ may not belong to $\cU^{s,x,\imath}_{ad}[s,\i)$ given $u_{\sc 1}(\cd)\oplus u_{\sc 2}(\cd),v_1(\cd)\oplus v_2(\cd)\in\cU^{s,x,\imath}_{ad}[s,\i)$. Therefore the classical calculation of variation method 
is not directly applicable for Problem (MF-LQ)$^*_{\sc\i}$.
To overcome this difficulty,   we  need to  derive some new equivalent forms for Problem (MF-LQ)$^*_{\scT}$,  Problem (MF-LQ)$^*_{\sc\i}$ and  Problem (MF-LQ)$^*_{\sc E}$ to remove such a dependence. To achieve this, we consider the following stabilizability condition which for \eqref{state}. 
 
	\begin{definition}\label{Def-stab-1} (1).  $(\Th_{\sc 1}(\cd),\Th_{\sc 2}(\cd)):\cM\mapsto\dbR^{m\times n}\times \dbR^{m\times n}$ is said to be a {\it stabilizer} for  the following system (with $\a(t)$ suppressed)
   \begin{align}\label{SDE10Th}\begin{cases}
		& dX_{\sc 1}(t)=(A_{\sc 1}+B_{\sc 1}\Th_{\sc 1}) X_{\sc 1}(t)dt +[(C_{\sc 1}+D_{\sc 1}{\Th_{\sc 1}})X_{\sc 1}(t) +(C_{\sc 2}+D_{\sc 2}\Th_{\sc 2}) X_{\sc 2}(t) ]dW(t),\\
	&  dX_{\sc 2}(t)= (A_{\sc 2}+B_{\sc 2}\Th_{\sc 2}) X_{\sc 2}(t)dt,\q t\in[s,\i),\\
	&  X_{\sc 1}(s)=x_{\sc 1},\q X_{\sc 2}(s)=x_{\sc 2}, \q \a(s)=\iota 	
    \end{cases}\end{align}
		admits a unique solution  $(X_{\sc 1}(\cd),X_{\sc 2}(\cd))\in L^2_{\dbF^\a}(s,\i;\dbR^n)^\perp\times L^2_{\dbF^\a}(s,\i;\dbR^n)$ for any $(s,x_1\oplus x_2,\imath)\in\cD$. %In this case we write $(X_{\sc 1}^0(\cd),X_{\sc 2}^0(\cd))$ by $(X_{\sc 1}^0(\cd;s,\iota,x_{\sc 1},\Th_{\sc 1},\Th_{\sc 2}),X_{\sc 2}^0(\cd;s,\iota,x_{\sc 2},\Th_{\sc 2}))$.
		
		(2)   $(\Th_{\sc 1}(\cd),\Th_{\sc 2}(\cd)):\cM\mapsto\dbR^{m\times n}\times \dbR^{m\times n}$ is said to be a {\it  dissipative strategy} of system \eqref{SDE10Th}  if there exist $\Si_{\sc 1},\Si_{\sc 2}:\cM\mapsto\dbS^n_{++}$ such that, for any $\jmath\in\cM$,
\begin{equation}\label{dissiptcre}
	\L[\Si_{\sc k}]+(A_{\sc k}+B_{\sc k}\Th_{\sc k})^\top \Si_{\sc k}+\Si_{\sc k}(A_{\sc k}+B_{\sc k}\Th_{\sc k}) +(C_{\sc k}+D_{\sc k}\Th_{\sc k})^\top\Si_{\sc 1}(C_{\sc k}+D_{\sc k}\Th_{\sc k})<0,
\end{equation}
for $k=1,2$.		
		
	\end{definition}	
		
It has been proved in   \cite{Mei-Wei-Yong-2025} that those two definitions are equivalent. Therefore, we write the set of all possible stabilizers by 
$\BS[A_{\sc 1},A_{\sc 2}, C_{\sc 1},C_{\sc 2}; B_{\sc 1},B_{\sc 2},D_{\sc 1},D_{\sc 2}].$
We now introduce the following assumption.\ms

{\noindent \bf (A2)}. $\BS[A_{\sc 1},A_{\sc 2}, C_{\sc 1},C_{\sc 2}; B_{\sc 1},B_{\sc 2},D_{\sc 1},D_{\sc 2}]\neq\emptyset$, or equivalently there exists a $(\widehat \Th_{\sc 1}(\cd),\widehat \Th_{\sc 2}(\cd))\in \BS[A_{\sc 1},A_{\sc 2}, C_{\sc 1},C_{\sc 2}; B_{\sc 1},B_{\sc 2},D_{\sc 1},D_{\sc 2}]$.
\ms

In fact,  it can be easily seen that  $(\Th_1(\cd),\Th_2(\cd))\in\BS[A_{\sc 1},A_{\sc 2}, C_{\sc 1},C_{\sc 2}; B_{\sc 1},B_{\sc 2},D_{\sc 1},D_{\sc 2}]$ if and only if 
\begin{align*}\begin{cases}
		& dX_{\sc 1}(t)=(A_{\sc 1}+B_{\sc 1}\Th_{\sc 1}) X_{\sc 1}(t)dt +[(C_{\sc 1}+D_{\sc 1}{\Th_{\sc 1}})X_{\sc 1}(t) ]dW(t),\\
	&  dX_{\sc 2}(t)= (A_{\sc 2}+B_{\sc 2}\Th_{\sc 2}) X_{\sc 2}(t)dt,\q t\in[s,\i),\\
	&  X_{\sc 1}(s)=x_{\sc 1},\q X_{\sc 2}(s)=x_{\sc 2}, \q \a(s)=\iota 	
    \end{cases}\end{align*}
    admits a unique solution  $(X_{\sc 1}(\cd),X_{\sc 2}(\cd))\in L^2_{\dbF^\a}(s,\i;\dbR^n)^\perp\times L^2_{\dbF^\a}(s,\i;\dbR^n)$ for any $(s,x_1\oplus x_2,i)\in\cD$.

Now let us adopt (A2) to remove the dependence of the admissible control set $\cU^{s,x,\imath}_{ad}[s,\i)$ on the initial value. For any  $u(\cd)=u_{\sc 1}(\cd)\oplus u_{\sc 2}(\cd)\in \cU^{s,x,\imath}_{ad}[s,\i)$, write the solution by 
$$X(\cd\,;s,x,\imath;u(\cd))=X_{\sc 1}(\cd\,;s,x,\imath;u(\cd))\oplus X_{\sc 2}(\cd\,;s,x,\imath;u(\cd)).$$
 Let 
 $$v_{\sc k}(t)=u_{\sc k}(t)-\widehat\Th_{\sc k}(\a(t))X_{\sc k}(t\,;s,x,\imath;u(\cd)),\text{ for all }t\geq s. $$
Then $v(\cd)=v_{\sc 1}(\cd)\oplus v_{\sc 2}(\cd)\in \cU[s,\infty)$ and 
$X(\cd\,;s,x,\imath;u(\cd))=\widehat X(\cd\,;s,x,\imath;v(\cd)$.

For any $v(\cd)=v_{\sc 1}(\cd)\oplus v_{\sc 2}(\cd)\in \cU[s,\infty)$, define 
$$u_{\sc k}(t)=\widehat\Th_{\sc k}(\a(t))\widehat X_{\sc k}(t\,;s,x,\imath;u(\cd))+v_{\sc k}(t) ,\text{ for all }t\geq s. $$
Then $u(\cd)= u_{\sc 1}(\cd)\oplus u_{\sc 2}(\cd)\in \cU_{ad}^{s,x,\imath}[s,\infty)$ and $X_k(\cd\,;s,x,\imath;u(\cd))=\widehat X_k(\cd\,;s,x,\imath;v(\cd))$.  Here $\widehat X_k(\cd\,;s,x,\imath;v(\cd))$ is the solution to 
\begin{align}\label{SDEhat}\begin{cases}
& d\widehat X_{\sc 1}(t)=\Big[(A_{\sc 1}+B_{\sc 1}\widehat\Th_{\sc 1}) \widehat X_{\sc 1}+B_{\sc 1}v_{\sc 1}+b_{\sc 1}\Big]dt\\
&	 \q+[(C_{\sc 1}+D_{\sc 1}\widehat\Th_{\sc 1}) \widehat X_{\sc 1}+(C_{\sc 2}+D_{\sc 2}\widehat\Th_{\sc 2})\widehat X_{\sc 2}+D_{\sc 1}v_{\sc 1}+D_{\sc 2}v_{\sc 2}+\si]dW,\\
& d\widehat X_{\sc 2}(t)=[(A_{\sc 2}+B_{\sc 2}\widehat\Th_{\sc 2})\widehat X_{\sc 2}+B_{\sc 2}v_{\sc 2}]dt, \q t\in[s,T],\\
& \widehat X_{\sc 1}(s)=x_{\sc 1},\q \widehat X_{\sc 2}(s)=x_{\sc 2}, \q \a(s)=\imath.\end{cases}\end{align}
We also define a new cost functional
\begin{align}\label{cost1hat}
&\widehat J_{\scT}(s, x_{\sc 1}\oplus x_{\sc 2},\imath;v_{\sc 1}(\cd)\oplus v_{\sc 2}(\cd))\nonumber\\
&\q=\sum_{k=1}^2\dbE  \int_s^T\[\lan  Q_{\sc k}X_{\sc k},X_{\sc k}\ran+2\lan S_{\sc k}X_{\sc k},\widehat \Th_{\sc k} X_{\sc k}+v_{\sc k}\ran+\lan R_{\sc k}(\widehat \Th_{\sc k} X_{\sc k}+v_{\sc k}),\widehat \Th_{\sc k} X_{\sc k}+v_{\sc k}\ran\nonumber\\
&\qq\qq\qq+\lan q_{\sc k}, X_{\sc k}\ran+\lan r_{\sc k}, \widehat\Th_{\sc k} X_{\sc k}+v_{\sc k}\ran\]dt.
\end{align}

Observed from above, Problem (MF-LQ)$_{\scT}$, Problem (MF-LQ)$_{\sc\i}$  and Problem (MF-LQ)$_{\sc E}$  can be further equivalently stated as follows.\ms

\noindent {\bf Problem (MF-LQ)$^{**}_{\sc T}$}.  For any $(s,\imath,x_{\sc 1}\oplus x_{\sc 2})\in\cD$, find a $\bar v_{\sc 1}(\cd)\oplus\bar v_{\sc 2}(\cd)\in\cU[s,T]$ such that
\begin{align*}\widehat J_{\scT}(s,x_{\sc 1}\oplus x_{\sc 2},\imath; \bar u_{\sc 1}(\cdot)\oplus \bar u_{\sc 2}(\cdot))=\inf_{u_{\sc 1}(\cdot)\oplus u_{\sc 2}(\cdot)\in \cU[s,T]}\widehat J_{\scT}(s,x_{\sc 1}\oplus x_{\sc 2},\imath; u_{\sc 1}(\cdot)\oplus u_{\sc 2}(\cdot)).\end{align*}

\noindent {\bf Problem (MF-LQ)$^{**}_{\sc\i}$}.  For any $(s,\imath,x_{\sc 1}\oplus x_{\sc 2})\in\cD$, find a $\bar v_{\sc 1}(\cd)\oplus\bar v_{\sc 2}(\cd)\in\cU[s,\i)$ such that
$$\widehat J_{\sc\i}(s,x_{\sc 1}\oplus x_{\sc 2},\imath; \bar u_{\sc 1}(\cdot)\oplus \bar u_{\sc 2}(\cdot))=\inf_{u_{\sc 1}(\cdot)\oplus u_{\sc 2}(\cdot)\in\cU[s,\i)}\widehat J_{\sc\i}(s,x_{\sc 1}\oplus x_{\sc 2},\imath; u_{\sc 1}(\cdot)\oplus u_{\sc 2}(\cdot)).$$

\noindent {\bf Problem (MF-LQ)$^{**}_{\sc E}$}.  For any $(s,\imath,x_{\sc 1}\oplus x_{\sc 2})\in\cD$, find a $\bar v_{\sc 1}(\cd)\oplus\bar v_{\sc 2}(\cd)\in\cU_{loc}[s,\i)$ such that
$$\widehat J_{\sc E}(s,x_{\sc 1}\oplus x_{\sc 2},\imath; \bar u_{\sc 1}(\cdot)\oplus \bar u_{\sc 2}(\cdot))=\inf_{u_{\sc 1}(\cdot)\oplus u_{\sc 2}(\cdot)\in\cU_{loc}[s,\i)}\widehat J_{\sc E}(s,x_{\sc 1}\oplus x_{\sc 2},\imath; u_{\sc 1}(\cdot)\oplus u_{\sc 2}(\cdot)).$$

At the same time (A2) reduces to 
$$(0,0)\in \BS[A_{\sc 1}+B_{\sc1}{\widehat\Th_{\sc1}},A_{\sc 2}+B_{\sc2}{\widehat\Th_{\sc2}}, C_{\sc 1},C_{\sc 2}; B_{\sc 1},B_{\sc 2},D_{\sc 1},D_{\sc 2}].$$
Without loss of generality, we assume that $\widehat\Th_{\sc 1}(\cd)=\widehat\Th_{\sc 2}(\cd)=0 $. Then (A2) can be represented as 
\ms

{\noindent \bf (A2)'} $(0,0)\in \BS[A_{\sc 1},A_{\sc 2}, C_{\sc 1},C_{\sc 2}; B_{\sc 1},B_{\sc 2},D_{\sc 1},D_{\sc 2}]$.
\ms

  In the sequel, we will consider  Problem (MF-LQ)$_{\scT}^*$,  Problem (MF-LQ)$_{\sc\i}^*$ and Problem (MF-LQ)$_{\sc E}^*$ under (A2)'.  Otherwise, we will work with Problem (MF-LQ)$_{\scT}^{**}$,  Problem (MF-LQ)$_{\sc\i}^{**}$ and Problem (MF-LQ)$_{\sc E}^{**}$. %It should noted that our results are independent of the choice $(\widehat\Th_{\sc1}(\cd),\widehat\Th_{\sc2}(\cd))$.

\subsection{Optimal Control for Problem (MF-LQ)$_{\scT}^*$}

Now we are ready to study the optimal control for Problem (MF-LQ)$_{\scT}^*$. We need the following positive-definiteness condition in the sequel.
\ms

{\noindent\bf (A3)}. For each $\imath\in\cM$ and $k=1,2$,
\begin{equation*}\label{Q}Q_{\sc k}(\imath)-S_{\sc k}(\imath)^\top R_{\sc k}(\imath)^{-1}S_{\sc k}(\imath)\in\dbS^n_{++}.\end{equation*}

Now we can state the results on  the optimal control of Problem (MF-LQ)$_{\scT}$.

\begin{theorem}\label{theorem1} Suppose  (A1), (A2)' and (A3) hold. Then  the following are true.

{\rm (i)} There exists a unique solution $P_{\sc 1,T}(\cd),P_{\sc 2,T}(\cd):[0,T]\times\cM\to\dbS_{++}^n$ to the following ARE:
\begin{align}\label{ARE00}
\begin{cases}
&\dot P_{\sc k,T}+\L[P_{\sc k,T}]+P_{\sc k,T}A_{\sc k}+A_{\sc k}^\top P_{\sc k,T} +C_{\sc k} ^\top P_{\sc k,T} C_{k,} +Q_{\sc k} \\
&\q-[P_{\sc k,T} B_{\sc k} +C_{\sc k} ^\top P_{\sc 1,T} D_{\sc k} +S_{\sc k} ^\top][R_{\sc k} +D_{\sc k} ^\top P_{\sc 1,T} D_{\sc k} ]^{-1}[B_{\sc k} ^\top P_{\sc k,T} +D_{\sc k} ^\top P_{\sc 1,T} C_{\sc k} +S_{\sc k} ]=0,\\
&P_{\sc k,T}(T)=0,\q R_{\sc k} +D_{\sc k} ^\top P_{\sc 1,T} D_{\sc k} >0,\qq\qq k=1,2
\end{cases}\end{align}
Write
$$\Th_{\sc k,T}(t,\imath)=-(R_{\sc k}+D_{\sc k}^\top P_{\sc 1,T}(t,\imath)D_{\sc k})^{-1}(B_{\sc k}^\top P_{\sc k,T}(t,\imath)+D_{\sc k}^\top P_{\sc 1,T}C_{\sc k}+S_{\sc k}).$$

{\rm (ii)} There exists a unique adapted solution $(\eta_{\sc1,T}(\cd),\zeta_{\scT}(\cd),\z^M_{\sc1,T}(\cd))\in L_{\dbF^\a}^2(0,\i;\dbR^n)^\perp\times L_{\dbF}^2(0,\i;\dbR^n)\times M_{\dbF^\a_-}^2(0,\i$; $\dbR^n)^\perp$ and
$(\eta_{\sc2,T}(\cd), \z^M_{\sc2,T}(\cd))\in L_{\dbF^\a}^2(0,\i;\dbR^n)\times M_{\dbF^\a_-}^2(0,\i;\dbR^n)$ to the following BSDE
\begin{align}\label{BSDE-Y10}\begin{cases}&d\eta_{\sc1,T}=\z _{\scT}dW+\z_{\sc1,T}^MdM -\Big((A_{\sc 1}^{\Th_{\sc1,T}})^\top \eta_{\sc 1,T}+(C_{\sc 1}^{\Th_{\sc1,T}})^\top\Pi_{\sc 1}[\zeta_{\scT}]+\varphi_{\sc 1,T}(t,\a(t))\Big)dt ,\\
& d\eta_{\sc2,T}=\z_{\sc2,T}^MdM-\Big( (A_{\sc 2}^{\Th_{\sc2,T}})^\top \eta_{\sc 2}
+(C_{\sc 2}^{\Th_{\sc2,T}})^\top\Pi_{\sc 2}[\z_{\scT}]+\varphi_{\sc 2,T}(t,\a(t))\Big)dt,\\
&\eta_{\sc1,T}(T)=\eta_{\sc2,T}(T)=0.\end{cases}\end{align} 
where
$\varphi_{\sc k,T}(t,\imath)=P_{\sc k,T}(t,\imath)b_{\sc k}(t)+(C_{\sc k}^{\Th_{\sc k,T}}(t,\imath))^\top P_{\sc 1,T}(t,\imath)\si_{\sc k}(t)+q_{\sc k}(t)+\Th_{\sc k,T}^\top(t,\imath) r_{\sc k}(t)$.
 Write
\begin{equation}\label{optimal strategy}v_{\sc k,T}(t,\imath)=-(R_{\sc k}+D_{\sc k}^\top P_{\sc 1,T}D_{\sc k})^{-1}(B_{\sc k}^\top \eta_{\sc k}+D_{\sc k}^\top\Pi_{\sc k}[\z_{\scT}]+D_{\sc k}^\top P_{\sc 1,T}\si_{\sc k}+r_{\sc k}),\q k=1,2.\end{equation}

{\rm (iii)} The optimal control of Problem {\rm(MF-LQ)$^*_{\scT}$} admits the following closed-loop representation, 
\begin{equation}\label{optimal control}\bar u_{\sc k,T}(t)=\Th_{\sc k,T}(t,\a(t))\bar X_{\sc k,T}(t)+v_{\sc k,T}(t,\a(t)),\qq k=1,2.\end{equation} 
Here $(\bar X_{\sc 1,T}(\cd),)\bar X_{\sc 2,T}(\cd))$  is the solution to 
\begin{align}\label{SDEopt}\begin{cases}
& d\bar X_{\sc 1,T}(t)=\Big[(A_{\sc 1}+B_{\sc 1}\Th_{\sc 1,T }) \bar X_{\sc 1,T}+B_{\sc 1}v_{\sc 1,T}+b_{\sc 1}\Big]dt\\
&	 \q+[(C_{\sc 1}+D_{\sc 1}\Th_{\sc 1,T}) \bar X_{\sc 1,T}+(C_{\sc 2}+D_{\sc 2}\Th_{\sc 2,T})\bar X_{\sc 2}+D_{\sc 1}v_{\sc 1,T}+D_{\sc 2}v_{\sc 2,T}+\si]dW,\\
& d\bar X_{\sc 2,T}(t)=[(A_{\sc 2}+B_{\sc 2}\Th_{\sc 2,T})\bar X_{\sc 2,T}+B_{\sc 2}v_{\sc 2,T}]dt, \q t\in[s,T],\\
& \bar X_{\sc 1,T}(s)=x_{\sc 1},\q \bar X_{\sc 2,T}(s)=x_{\sc 2}, \q \a(s)=\imath.\end{cases}\end{align}

\end{theorem}

Until now, we have presented the expilcit form of the optimal control for Problem (MF-LQ)$_{\scT}$. The rest of the paper is focused on the asymptotic behavior of the optimal control in \eqref{optimal control}. Before finishing this section, let us make the following remark.

\begin{remark} (1) Note Theorem \ref{theorem1} also holds if (A2)' is replaced by (A2). In particular, it is worth to emphasize that the optimal control in \eqref{optimal control} is independent of the choice of $(\widehat\Th_1(\cd),\widehat\Th_2(\cd))$ in  Problem (MF-LQ)$^{**}_{\sc T}$. For more details, one check \cite{Mei-Wei-Yong-2025}. \ms

(2) The assumption (A3) can be possibly weaken by some uniform convexity assumption on the cost functional. This paper will not consider this part.
\end{remark}

\section{Asymptotic Behavior of the Optimal Controls}\label{sec:abo}
With the closed-loop representation of the optimal control in \eqref{optimal control}, 
this section is devoted to studying the asymptotic behavior as $T\rightarrow\infty$. We will consider $\Th_{\sc k,T}(\cd)$ and $v_{\sc k,T}(\cd)$ separately.
\subsection{Riccati Equation}

To study the asymptotic behavior of $(P_{\sc 1,T}(\cd),P_{\sc 2,T}(\cd))$ as $T\rightarrow\i$, we consider the following ARE
\begin{align}\label{ARE00int}
\begin{cases}
&\L[P_{\sc k,\i}]+P_{\sc k,\i}A_{\sc k}+A_{\sc k}^\top P_{\sc k,\i} +C_{\sc k} ^\top P_{\sc k,\i} C_{k,} +Q_{\sc k} \\
&\q-[P_{\sc k,\i} B_{\sc k} +C_{\sc k} ^\top P_{\sc 1,\i} D_{\sc k} +S_{\sc k} ^\top][R_{\sc k} +D_{\sc k} ^\top P_{\sc 1,\i} D_{\sc k} ]^{-1}[B_{\sc k} ^\top P_{\sc k,\i} +D_{\sc k} ^\top P_{\sc 1,\i} C_{\sc k} +S_{\sc k} ]=0.\\
& R_{\sc k} +D_{\sc k} ^\top P_{\sc 1,T} D_{\sc k} >0,\text{ for }k=1,2.\end{cases}\end{align}
Define
$$\Th_{\sc k,\i}(\imath)=-(R_{\sc k}+D_{\sc k}^\top P_{\sc 1,\i}(t,\imath)D_{\sc k})^{-1}(B_{\sc k}^\top P_{\sc k,\i}(t,\imath)+D_{\sc k}^\top P_{\sc 1,\i}C_{\sc k}+S_{\sc k}).$$

The following proposition presents the convergence of $(P_{\sc 1,T}(\cd),P_{\sc 2,T}(\cd))$.

\begin{proposition} Suppose (A1), (A2)' and (A3) hold.  The following are true.\ms 

{\rm (i)}. The ARE \eqref{ARE00int} admits a unique solution $(P_1(\cd),P_2(\cd)):\cM\mapsto\dbS^n_{++}$ such that 
$(\Th_{\sc 1,\i}(\cd),\Th_{\sc 2,\i}(\cd))\in\BS[A_{\sc 1},A_{\sc 2}, C_{\sc 1},C_{\sc 2}; B_{\sc 1},B_{\sc 2},D_{\sc 1},D_{\sc 2}].$

{\rm (ii)}. For any given $t\in[0,\infty)$,  the following convergence holds
\begin{equation}\label{P to P}P_{\sc k,T}(t,\imath)=P_{\sc k,T-t}(0,\imath)\nearrow P_{\sc k,\i}(\imath),\qq\hb{as } T\nearrow\i,\q\forall\imath\in\cM.\end{equation}

{\rm (iii)}. There exists a $\d_*>0$ and $K>0$ (independent of $T$) so that
\bel{P-P}
0\les P_{\sc k,\i}(\imath)-P_{\sc k, T}(t,\imath)\les Ke^{-\d_*(T-t)}I,\qq t\in[0,T].\ee
Consequently,
\bel{Q-Q}
|\Th_{\sc k,\i}(\imath)-\Th_{\sc k, T}(t,\imath)|\les Ke^{-\d_*(T-t)}I,\qq t\in[0,T].\ee

\end{proposition}

\begin{proof}
(i) and (ii) have been proved in \cite{Mei-Wei-Yong-2025}.

(iii). Because of \eqref{P-P}, we know that $\Th_{\sc k,T}(t,\imath)=\Th_{\sc k,T-t}(0,\imath)\rightarrow\Th_{\sc k,\i}(\imath)$ as $T\rightarrow\infty.$ Moreover, there exists $\Si_{\sc k}(\cd):\cM\rightarrow\dbS^n_{++}$ such that 
\begin{align}&\L[\Si_{\sc k}]+(A_{\sc k}+B_{\sc k}\Th_{\sc k,\i})^\top \Si_{\sc k}+\Si_{\sc k}(A_{\sc k}+B_{\sc k}\Th_{\sc k,\i}) \nonumber\\
&\q+(C_{\sc k}+D_{\sc k}\Th_{\sc k,\i})^\top\Si_{\sc 1}(C_{\sc k}+D_{\sc k}\Th_{\sc k,\i})\leq -\d_*\Si_{\sc k}.\label{ly1}\end{align}
Therefore, there exists a $t_0>0$ independent of $T$ such that for 
\begin{align}&\L[\Si_{\sc k}]+(A_{\sc k}+B_{\sc k}\Th_{\sc k,T}(t,\imath))^\top \Si_{\sc k}+\Si_{\sc k}(A_{\sc k}+B_{\sc k}\Th_{\sc k,T}(t,\imath))\nonumber\\
&\q +(C_{\sc k}+D_{\sc k}\Th_{\sc k,T}(t,\imath))^\top\Si_{\sc 1}(C_{\sc k}+D_{\sc k}\Th_{\sc k,T}(t,\imath))\leq -\frac{\d_*}2\Si_{\sc k},\label{ly2}\end{align}
for all $t\in[s,T-t_0] $  and  for all $T>s+t_0$.

Now we consider the homogeneous case of Problem (MF-LQ)$_{\scT}^*$ and Problem (MF-LQ)$_{\sc\i}^*$ and we write $J_{\scT}^0(s,x_{\sc 1}\oplus x_{\sc 2};u_{\sc 1}(\cd)\oplus u_{\sc 2}(\cd))$ and $J_{\sc\i}^0(s,x_{\sc 1}\oplus x_{\sc 2};u_{\sc 1}(\cd)\oplus u_{\sc 2}(\cd))$ by the corresponding cost functionals. Let $(\bar X^0_{\sc 1,T}(\cd),\bar X^0_{\sc 2,T}(\cd))$ be the solution to \eqref{SDEopt} with $v_{\sc k,T}(\cd), b(\cd),\sigma(\cd)=0$ which is the optimal state process for  Problem (MF-LQ)$_{\scT}^*$. Now applying It\^o's formula on 
$$t\mapsto \sum_{k=1}^2\lan \Si_{\sc k}(\a(t)) \bar X^0_{\sc k,T}(t),\bar X^0_{\sc k,T}(t) \ran,$$
\eqref{ly2} yields that  for $t\in[s,T-t_0]$,
\begin{align*}
\frac d{dt}\dbE\sum_{k=1}^2\lan \Si_{\sc k}(\a(t)) \bar X^0_{\sc k,T}(t),\bar X^0_{\sc k,T}(t) \ran\leq -\frac{\d_*}2\dbE\sum_{k=1}^2\lan \Si_{\sc k}(\a(t)) \bar X^0_{\sc k,T}(t),\bar X^0_{\sc k,T}(t) \ran.
\end{align*}
Grownwall's inequality implies that 
\begin{align*}
\dbE\sum_{k=1}^2\lan \Si_{\sc k}(\a(t)) \bar X^0_{\sc k,T}(t),\bar X^0_{\sc k,T}(t) \ran\leq K e^{-\frac{\d_*}2(t-s)}|x|^2, \text{ for }t\in[s,T-t_0].
\end{align*}

For $t\in[T-t_0,T]$, due to the boundedness of $A_{\sc k}, B_{\sc k}, C_{\sc k}, D_{\sc k}, \Th_{\sc k}$, it follows that
\begin{align*}
\frac d{dt}\dbE\sum_{k=1}^2\lan \Si_{\sc k}(\a(t)) \bar X^0_{\sc k,T}(t),\bar X^0_{\sc k,T}(t) \ran\leq K\dbE\sum_{k=1}^2\lan \Si_{\sc k}(\a(t)) \bar X^0_{\sc k,T}(t),\bar X^0_{\sc k,T}(t) \ran.
\end{align*}
Grownwall's inequality implies that 
\begin{align*}
&\dbE\sum_{k=1}^2\lan \Si_{\sc k}(\a(t)) \bar X^0_{\sc k,T}(t),\bar X^0_{\sc k,T}(t) \ran\\
&\leq K e^{K(t-(T-t_0))}\dbE\sum_{k=1}^2\lan \Si_{\sc k}(\a(T-t_0)) \bar X^0_{\sc k,T}(T-t_0),\bar X^0_{\sc k,T}(T-t_0) \ran\\
&\leq K e^{K(t-(T-t_0)}e^{-\frac{\d_*}2(T-t_0-s)}|x|^2\leq Ke^{-\frac{\d_*}2(t-s)}|x|^2,\text{ for }t\in[T-t_0,T].
\end{align*}
This is to say
\begin{align}\label{XTbound}
\dbE\sum_{k=1}^2|\bar X^0_{\sc k,T}(t)|^2\leq Ke^{\frac{\d_*}2(t-s)}|x|^2,\text{ for any }t\in[s,T].
\end{align}

Now let us prove \eqref{P-P}. By the dynamic programming principle and \eqref{XTbound}, we have
\begin{align*}
&\sum_{k=1}^2\dbE\lan P_{\sc k,T}(s,\imath) x_{\sc k},x_{\sc k}\ran= J^0_{\sc T}(s,x_{\sc 1}\oplus x_{\sc 2},\imath; \bar u_{\sc 1,T}(\cd)\oplus\bar u_{\sc 2,T}(\cd))\\
&=J^0_{\sc T}(s,x_{\sc 1}\oplus x_{\sc 2},\imath; \bar u_{\sc 1,T}(\cd)\oplus\bar u_{\sc 2,T}(\cd))+\dbE \sum_{k=1}^2\lan P_{\sc k,\i}(\a(T)) \bar X^0_{\sc k,T}(T),\bar X^0_{\sc k,T}(T)\ran\\
&\q-\dbE \sum_{k=1}^2\lan P_{\sc k,\i}(\a(T)) \bar X_{\sc k,T}(T),\bar X^0_{\sc k,T}(T)\ran\\
&\geq J^0_{\sc \i}(s,x_{\sc 1}\oplus x_{\sc 2},\imath; \bar u_{\sc 1,\i}(\cd)\oplus\bar u_{\sc 2,\i}(\cd))-Ke^{-\frac{\d_*}2(T-s)}|x|^2\\
&=\sum_{k=1}^2\dbE\lan P_{\sc k,\i}(s,\imath) x_{\sc k},x_{\sc k}\ran-Ke^{-\frac{\d_*}2(T-s)}|x|^2
\end{align*}
By the arbitrariness of $x_1\oplus x_2\in L^2_{\cF_s}(\dbR^n)$, we have \eqref{P-P}. \eqref{Q-Q} follows from the definition of $\Th_{\sc k, T}$ and $\Th_{\sc k,\i}$ and the uniform boundedness of $R_{\sc k}+D_k^\top P_{\sc 1,T}D_k>0$ from below.

\end{proof}

\subsection{BSDEs}
In this subsection, we will consider the asymptotic behavior of $v_{\sc k,T}(\cd)$. The main effort is devoted to studying  BSDEs \eqref{BSDE-Y10} as $T\rightarrow\i$.
Write 
$$\xi(t):=\dbE\(|b(t)|^2+|\si(t)|^2+|q(t)|^2+|r(t)|^2+|\bar q(t)|^2+|\bar r(t)|^2\).$$
We need the following assumption to study the asymptotic behavior the BSDE \eqref{BSDE-Y10}.\ms

\noindent{\bf (A3)}  It follows that 
\begin{equation}\label{int<i}\sup_{r\in[0,\i)}\int_0^\i e^{-{\d_*\over4}|r-t|}\xi(t)dt<\i.\end{equation}
\ms

Let $T\rightarrow\infty$ in \eqref{BSDE-Y10}, it is natural to arise the following BSDE over $[0,\i).$

\begin{align}\label{BSDE-INT}\begin{cases}&d\eta_{\sc1,\i }=\z _{\sc\i}dW+\z_{\sc1,\i }^MdM-\Big((A_{\sc 1}^{\Th_{\sc 1,\i}})^\top \eta_{\sc 1,\i}+(C_{\sc 1}^{\Th_{\sc 1,\i}})^\top\Pi_{\sc 1}[\zeta_{\sc\i}]\)dt\\
&\qq\qq-\(P_{\sc 1,\i}b_{\sc 1}+(C_{\sc 1}^{\Th_{\sc 1,\i}})^\top P_{\sc 1,\i}\si_{\sc 1}+q_{\sc 1}+\Th_{\sc 1}^\top r_{\sc 1}\Big)dt ,\\
& d\eta_{\sc2,\i }=\z_{\sc2,\i }^MdM- (A_{\sc 2}^{\Th_{\sc 2,\i}})^\top \eta_{\sc 2,\i}dt\\
&\qq\qq-\((C_{\sc 2}^{\Th_{\sc 2,\i}})^\top\Pi_{\sc 2}[\z_{\sc\i}]+P_{\sc 2,\i}b_{\sc 2}+(C_{\sc 2}^{\Th_{\sc 2}})^\top P_{\sc 1,\i}\si_{\sc 2}+q_{\sc 2}+\Th_{\sc 2}^\top r_{\sc 2}\Big)dt.\end{cases}\end{align}
We have the following proposition.

\begin{proposition}\label{propBSDE}
{\rm (i)} The BSDE \eqref{BSDE-INT} admits a unique solution solution $$(\eta_{\sc1,T}(\cd),\zeta_{\scT}(\cd),\z^M_{\sc1,T}(\cd))\in L_{\dbF^\a}^{2}(0,T;\dbR^n)^\perp\times L_{\dbF}^2(0,T;\dbR^n)\times M_{\dbF^\a_-}^2(0,\i;\dbR^n)^\perp$$ and
$$(\eta_{\sc2,T}(\cd), \z^M_{\sc2,T}(\cd))\in L_{\dbF^\a}^{2}(0,T;\dbR^n)\times M_{\dbF^\a_-}^{2}(0,T;\dbR^n).$$

{\rm (ii)} The BSDE \eqref{BSDE-INT} admits a unique solution solution $$(\eta_{\sc1,\i}(\cd),\zeta_{\sc\i}(\cd),\z^M_{\sc1,\i}(\cd))\in L_{\dbF^\a}^{2,loc}(0,\i;\dbR^n)^\perp\times L_{\dbF}^{2,loc}(0,\i;\dbR^n)\times M_{\dbF^\a_-}^{2,loc}(0,\i;\dbR^n)^\perp$$ and
$$(\eta_{\sc2,\i}(\cd), \z^M_{\sc2,\i}(\cd))\in L_{\dbF^\a}^{2,loc}(0,\i;\dbR^n)\times M_{\dbF^\a_-}^{2,loc}(0,\i;\dbR^n).$$
Here $L_{\dbH}^{2,loc}(0,\i;\dbR^n)=\cap_{T>0}L_{\dbH}^{2}(0,T;\dbR^n), M_{\dbH_-}^{2,loc}(0,\i;\dbR^n)=\cap_{T>0}M_{\dbH_-}^{2}(0,T;\dbR^n)$, for $\dbH=\dbF,\dbF^\a$.\ms

{\rm (iii)}
For $k=1,2$, we have
\begin{align}&\sum_{k=1}^2\(\dbE|\eta_{\sc k,T}(t)|^2+\dbE\int_t^T e^{-{\d_*\over4}(s-t)}\sum_{\jmath\neq\imath}\lambda_{\imath\jmath}|\z^M_{\sc k,T}(s,\jmath)|^2{\bf 1}_{[\a(s)=\imath]}ds\nonumber\\
&\q+\dbE\int_t^T e^{-{\d_*\over4}(s-t)}|\z_{\sc k,T}(s)|^2ds\)\les K\int_t^T e^{-{\d_*\over4}(s-t)}\xi(s)ds.\label{stabsde}
\end{align}
It also holds that
\begin{align}
&\sum_{k=1}^2\(\dbE|\eta_{\sc k,T}(t)-\eta_{\sc k,\i}(t)|^2+\dbE\int_t^{T}e^{-{\d_*\over4}(s-t)}\sum_{\jmath\neq\imath}\lambda_{\imath\jmath}|\z^M_{\sc k,T}(s,\jmath)-\z^M_{\sc k,\i}(s,\jmath)|^2{\bf 1}_{[\a(s)=\imath]}ds\nonumber\\
&\q+\dbE\int_t^{T}
e^{-{\d_*\over4}(s-t)}|\z_{\sc k,T}(s)-\z_{\sc k,\i}(s)|^2ds\)\les Ke^{-{\d_*\over8}(T-t)}\int_t^{\sc\i} e^{-{\d_*\over4}(s-t)}\xi(s)ds.\label{E[eta-eta]}
\end{align}

{\rm (iv)} The following are true.

\begin{align}
&\dbE\int_0^t e^{-{\d_*\over4}(t-s)}[|\z_{\sc k,T}(s)|^2+|\z_{\sc k,\i}(s)|^2]dt\les K\int_0^\i\3n e^{-{\d_*\over4}|t-s|}\xi(s)ds,\label{festimate1}\\
&\dbE\int_0^t e^{-{\d_*\over4}(t-s)}|\z_{ k,\scT}(s)-\z_{\sc k,i}(s)|^2ds\les K e^{-\frac\d 8(T-t)}\int_0^\i  e^{-{\d_*\over4}|t-s|}\xi(s)ds,\label{festimate2}\\
%&\dbE\int_0^te^{-\d(t-s)}| v_{\i}(s)|^2ds\les K\int_0^\i e^{-{\d_*\over4}|t-r|}\xi(r)dr.\\
&\dbE\int_0^te^{-\frac\d4(t-s)}| v_{\sc k,T}(s)-v_{\sc k,\i}(s)|^2 ds\les Ke^{-{\d_*\over8}(T-t)}\int_0^\i e^{-{\d_*\over4}|t-s|}\xi(s)ds\label{v-v},\\
&\dbE\int_0^T|\zeta_{\sc k,T}(s)|^2+|\zeta_{\sc k,\i}(s)|^2dt\les K\int_0^T
\xi(s)ds+K (T+1)\sup_{s\ges 0}\int_{0}^\infty e^{-\frac\d 4|s-r|}\xi(r)dr,\label{XXb}
\\
&\dbE\int_0^T|\zeta_{\sc k,T}(s)-\zeta_{\sc k,\i}(s)|^2ds\les K\sup_{s\ges 0}\int_{0}^\infty e^{-\frac\d 4|s-r|}\xi(r)dr,\label{XXb2}\\
&\dbE\big[|\bar X^{x,\imath}_{\sc k,T}(t)|^2+|\bar X^{x,\imath}_{\sc k,\i}(t)|^2\big]\les K \(e^{-{\d_*\over2}t}|x|^2+\int_0^\i e^{-{\d_*\over4}|t-s|}\xi(s)ds\).\label{boundEXTX}
\end{align}
\end{proposition}

\begin{proof} The results here are parallel to Proposition 3.5 and Proposition 3.7 in \cite{Mei-Wang-Yong-2025b} where the key difference lies in the mean-field terms in \eqref{BSDE-Y10} and \eqref{BSDE-INT}.
To tackle this, we will propose two BSDEs without mean-field terms so that the the solutions to \eqref{BSDE-Y10} and \eqref{BSDE-INT} can be repsented using mappings $\Pi_1$ and $\Pi_2.$

We consider the following two BSDEs:
\begin{align}\label{BSDE-INT2}\begin{cases}&d\check\eta_{\sc1,T}(t)=\check\z_{\scT}dW+\check\z_{\sc1,T}^MdM \\
&\q-\Big((A_{\sc 1}^{\Th_{\sc1,T}})^\top \check\eta_{\sc 1,T}+(C_{\sc 1}^{\Th_{\sc1,T}})^\top\check\zeta_{\scT} +P_{\sc 1,T}b_{\sc 1}+(C_{\sc 1}^{\Th_{\sc 1,\i}})^\top P_{\sc 1,T}\si_{\sc 1}+q_{\sc 1}+\Th_{\sc 1}^\top r_{\sc 1}\Big)dt\\
&d\check\eta_{\sc2,T }=\check\z_{\sc2,T }^MdM-(A_{\sc 2}^{\Th_{\sc 2,T}})^\top \check\eta_{\sc 2,T}\\
&\q-\Big( (C_{\sc 2}^{\Th_{\sc 2,T}})^\top\Pi_{\sc 2}[\check\z_{\scT}]
+P_{\sc 2,T}b_{\sc 2}+(C_{\sc 2}^{\Th_{\sc 2}})^\top P_{\sc 1,T}\si_{\sc 2}+q_{\sc 2}+\Th_{\sc 2,T}^\top r_{\sc 2}\Big)dt
 \end{cases}
\end{align}
and 
\begin{align}\label{BSDE-INT3}\begin{cases} &d\check\eta_{\sc1,\i }=\check\z _{\sc\i}dW+\check\z_{\sc1,\i }^MdM\\
&\q -\Big((A_{\sc 1}^{\Th_{\sc 1,\i}})^\top \check\eta_{\sc 1,\i}+(C_{\sc 1}^{\Th_{\sc 1,\i}})^\top\check\zeta_{\sc\i}+P_{\sc 1,\i}b_{\sc 1}+(C_{\sc 1}^{\Th_{\sc 1,\i}})^\top P_{\sc 1,\i}\si_{\sc 1}+q_{\sc 1}+\Th_{\sc 1}^\top r_{\sc 1}\Big)dt.\\
&d\check\eta_{\sc2,\i }=\check\z_{\sc2,\i }^MdM-(A_{\sc 2}^{\Th_{\sc 2,\i}})^\top \check\eta_{\sc 2,T}\\
&\qq-\Big((C_{\sc 2}^{\Th_{\sc 2,\i}})^\top\Pi_{\sc 2}[\check\z_{\sc\i}]
+P_{\sc 2,\i}b_{\sc 2}+(C_{\sc 2}^{\Th_{\sc 2}})^\top P_{\sc 1,\i}\si_{\sc 2}+q_{\sc 2}+\Th_{\sc 2}^\top r_{\sc 2}\Big)dt\\
\end{cases}
\end{align}
For \eqref{BSDE-INT2}, we will solve the first BSDE first. Seeing $\check\zeta_{\scT}(\cdot)$ as a  given stochastic process, one then proceed with the second BSDE. In this case, \eqref{BSDE-INT2} is essentially a regular BSDE without mean-field terms.  The similar idea can be applied to \eqref{BSDE-INT3}. Observed from this, we are allowed to apply the previous results  in \cite{Mei-Wang-Yong-2025b} on \eqref{BSDE-INT2} and \eqref{BSDE-INT3}.

By Proposition 3.5 in \cite{Mei-Wang-Yong-2025b}, \eqref{BSDE-INT2} and \eqref{BSDE-INT3} admits a unique solution
\begin{align*}
&(\check\eta_{\sc1,T}(\cd),\check\zeta_{\scT}(\cd),\check\z^M_{\sc1,T}(\cd))\in L_{\dbF}^{2}(0,T;\dbR^n)\times L_{\dbF}^{2}(0,T;\dbR^n)\times M_{\dbF_-}^{2}(0,T;\dbR^n),\\
&(\check\eta_{\sc2,T}(\cd), \check\z^M_{\sc2,T}(\cd))\in L_{\dbF^\a}^{2}(0,T;\dbR^n)\times M_{\dbF^\a_-}^{2}(0,T;\dbR^n),\\
&(\check\eta_{\sc1,\i}(\cd),\check\zeta_{\sc\i}(\cd),\check\z^M_{\sc1,\i}(\cd))\in L_{\dbF}^{2,loc}(0,\i;\dbR^n)\times L_{\dbF}^{2,loc}(0,\i;\dbR^n)\times M_{\dbF_-}^{2,loc}(0,\i;\dbR^n),\\
&(\check\eta_{\sc2,\i}(\cd), \check\z^M_{\sc2,\i}(\cd))\in L_{\dbF^\a}^{2,loc}(0,\i;\dbR^n)\times M_{\dbF^\a_-}^{2,loc}(0,\i;\dbR^n).\end{align*}
Through the orthogonal decomposition, it can be verified that 
\begin{align*}
&(\Pi_1[\check\eta_{\sc1,T}](\cd),\check\zeta_{\scT}(\cd),\Pi_1[\check\z^M_{\sc1,T}](\cd))\in L_{\dbF}^{2}(0,T;\dbR^n)\times L_{\dbF}^{2}(0,T;\dbR^n)\times M_{\dbF_-}^{2}(0,T;\dbR^n),\\
&(\check\eta_{\sc2,T}(\cd), \check\z^M_{\sc2,T}(\cd))\in L_{\dbF^\a}^{2}(0,T;\dbR^n)\times M_{\dbF^\a_-}^{2}(0,T;\dbR^n)\end{align*}
is the solution to \eqref{BSDE-Y10} and 
\begin{align*}
&(\Pi_1[\check\eta_{\sc1,\i}](\cd),\check\zeta_{\sc\i}(\cd),\Pi_1[\check\z^M_{\sc1,\i}](\cd))\in L_{\dbF}^{2,loc}(0,\i;\dbR^n)\times L_{\dbF}^{2,loc}(0,\i;\dbR^n)\times M_{\dbF_-}^{2,loc}(0,\i;\dbR^n),\\
&(\check\eta_{\sc2,\i}(\cd), \check\z^M_{\sc2,\i}(\cd))\in L_{\dbF^\a}^{2,loc}(0,\i;\dbR^n)\times M_{\dbF^\a_-}^{2,loc}(0,\i;\dbR^n).\end{align*}
 is the solution to \eqref{BSDE-INT}. Provided the estimates in  Proposition 3.7 in  \cite{Mei-Wang-Yong-2025b},   \eqref{festimate1}-\eqref{boundEXTX} hold. The proof is complete.
\end{proof}

\section{Strong Turnpike Property}\label{sec:stp}
Now we are ready to prove our main results on the strong turnpike property. Without loss of generality, we assume $s=0$ in the sequel.
Recall that the optimal state process for Problem (MF-LQ)$_{\scT}$ satisfies
\begin{align}\label{optsys1}
\begin{cases}
& d\bar X_{\sc 1,T}(t)=\Big[(A_{\sc 1}+B_{\sc 1}\Th_{\sc 1,T }) \bar X_{\sc 1,T}+B_{\sc 1}v_{\sc 1,T}+b_{\sc 1}\Big]dt\\
&	 \q+\Big[(C_{\sc 1}+D_{\sc 1}\Th_{\sc 1,T}) \bar X_{\sc 1,T}+(C_{\sc 2}+D_{\sc 2}\Th_{\sc 2,T})\bar X_{\sc 2}+D_{\sc 1}v_{\sc 1,T}+D_{\sc 2}v_{\sc 2,T}+\si\Big]dW,\\
& d\bar X_{\sc 2,T}(t)=\Big[(A_{\sc 2}+B_{\sc 2}\Th_{\sc 2,T})\bar X_{\sc 2,T}+B_{\sc 2}v_{\sc 2,T}\Big]dt, \q t\in[0,T],\\
& \bar X_{\sc 1,T}(0)=x_{\sc 1},\q \bar X_{\sc 2,T}(0)=x_{\sc 2}, \q \a(0)=\imath.\end{cases}
\end{align}

To define the limit process, we consider the following  control 
\begin{align}\label{optintcont}\bar u_{\sc k,\i}(t)=\Th_{\sc k,\i}(\a(t)) \bar X_{\sc k,\i}(t)+v_{\sc k,\i}(t,\a(t)).\end{align}
Then the state process $\bar X_{\sc 1,\i}(\cd)\oplus \bar X_{\sc 2,\i}(\cd)$ satisfies the following SDE
\begin{align}\label{optsys2}
\begin{cases}
& d\bar X_{\sc 1,\i}(t)=\Big[(A_{\sc 1}+B_{\sc 1}\Th_{\sc 1,\i }) \bar X_{\sc 1,\i}+B_{\sc 1}v_{\sc 1,\i}+b_{\sc 1}\Big]dt\\
&	 \q+\Big[(C_{\sc 1}+D_{\sc 1}\Th_{\sc 1,\i}) \bar X_{\sc 1,\i}+(C_{\sc 2}+D_{\sc 2}\Th_{\sc 2,\i})\bar X_{\sc 2}+D_{\sc 1}v_{\sc 1,\i}+D_{\sc 2}v_{\sc 2,\i}+\si\Big]dW,\\
& d\bar X_{\sc 2,\i}(t)=\Big[(A_{\sc 2}+B_{\sc 2}\Th_{\sc 2,\i})\bar X_{\sc 2,\i}+B_{\sc 2}v_{\sc 2,\i}\Big]dt, \q t\in[0,\i),\\
& \bar X_{\sc 1,\i}(0)=x_{\sc 1},\q \bar X_{\sc 2,\i}(0)=x_{\sc 2}, \q \a(0)=\imath.\end{cases}
\end{align}

We will first present  our main result on the Turnpike property in the paper. Then we will verify the optimality of the  control in \eqref{optintcont} for Problem (MF-LQ)$_{\sc \i}$ or Problem (MF-LQ)$_{\sc E}$ under different assumptions.
\begin{theorem}\label{mainthm} Suppose (A1), (A2) and (A3) hold. Then there exist absolute constants $\beta, K>0$ independent of $(t,T)$ such that 
 \begin{align*}
&\sum_{k=1}^2\dbE\Big(|\bar X_{\sc k,T}^{0,x,\imath}(t)-\bar X_{\sc k,\i}^{0,x,\imath}(t)|^2+\int_0^t e^{-\frac{\d_*}4(t-r)}|\bar u_{\sc k,T}^{0,x,\imath}(r)-\bar u_{\sc k,\i}^{0,x,\imath}(r)|^2dr\Big)\\
&\q\les 
Ke^{-\frac {\d_*} 8(T-t)}\(e^{-\frac{\d_*}4t}|x|^2+\int_0^\i e^{-{\d_*\over4}|t-r|}\xi(r)dr\),\end{align*}
for all  $t\in[0,T]$.
\end{theorem}
\begin{proof}
In the proof, the top index $(0,x,\imath)$, $t$ and $\a(t)$ are suppressed. By \eqref{optsys1} and \eqref{optsys2}, it follows that
\begin{align*}
&d(\bar X_{\sc 1,T}(t)-\bar X_{\sc 1,\i}(t))=\Big[(A_{\sc 1}+B_{\sc 1}\Th_{\sc1,\i}) (\bar X_{\sc1,T}-\bar X_{\sc1,\i})+B_1(v_{\sc1,T}-v_{\sc1,\i})\Big]dt\\
&\q+\Big[(C_{\sc 1}+D_{\sc 1}\Th_{\sc1,\i})(\bar X_{\sc1,T}-\bar X_{\sc1,\i})+(C_{\sc 2}+D_{\sc 2}\Th_{\sc 2,\i})(\bar X_{\sc2,T}-\bar X_{\sc2,\i})\]dW\\
&\qq+B_{\sc 1}(\Th_{\sc1,T}-\Th_{\sc1,\i})\bar X_{\sc1,T}dt+\Big[D_{\sc 1}(\Th_{\sc1,T}-\Th_{\sc1,\i})\bar X_{\sc1,T}+D_{\sc 2}(\Th_{\sc 2,T}-\Th_{\sc2,\i})\bar X_{\sc2,T} \Big]dW,\\
&\qq+\Big[D_{\sc 1}(v_{\sc 1,T}-v_{\sc 1,\i})+D_{\sc 2}(v_{\sc 2,T}-v_{\sc 2,\i})\Big]dW\\
&d(\bar X_{\sc 2,T}(t)-\bar X_{\sc 2,\i}(t))=\[(A_{\sc 2}+B_{\sc 2}\Th_{\sc2,\i}) (\bar X_{\sc2,T}-\bar X_{\sc2,\i})\]dt\\
&\qq+\Big[B_{\sc 2}(\Th_{\sc2,T}-\Th_{\sc2,\i})\bar X_{\sc2,T}+B_{\sc 2}(v_{\sc 2,T}-v_{\sc 2,\i})\]dt
\end{align*}
The applying It\^o's formula on 
$$t\mapsto \sum_{k=1}^2\lan P_{\sc k,\i} (\a(t))(\bar X_{\sc k,T}(t)-\bar X_{\sc k,\i}(t)),\bar X_{\sc k,T}(t)-\bar X_{\sc k,\i}(t)\ran,$$
\eqref{Q-Q} and \eqref{ly1} yield that 
\begin{align*}
&\frac {d}{dt}\dbE \sum_{k=1}^2\lan P_{\sc k,\i}(\a(t)) (\bar X_{\sc k,T}(t)-\bar X_{\sc k,\i}(t)),\bar X_{\sc k,T}(t)-\bar X_{\sc k,\i}(t)\ran\\
&\leq -\frac{\d_*}2\dbE \sum_{k=1}^2\lan P_{\sc k,\i}(\a(t)) (\bar X_{\sc k,T}(t)-\bar X_{\sc k,\i}(t)),\bar X_{\sc k,T}(t)-\bar X_{\sc k,\i}(t)\ran\\
&\qq+Ke^{-\frac {\d_*}4(T-t)}\sum_{k=1}^2\dbE|\bar X_{\sc k,T}(t)|^2+K\sum_{k=1}^2\dbE|v_{\sc k,T}-v_{\sc k,\i}|^2
\end{align*}
Using \eqref{v-v}, Grownwall's inequality implies that
\begin{align*}
&\dbE \sum_{k=1}^2|\bar X_{\sc k,T}(t)-\bar X_{\sc k,\i}(t))|^2\\
&\leq K\dbE \sum_{k=1}^2\lan P_{\sc k,\i}(\a(t)) (\bar X_{\sc k,T}(t)-\bar X_{\sc k,\i}(t)),\bar X_{\sc k,T}(t)-\bar X_{\sc k,\i}(t)\ran\\
&\leq K\int_0^te^{-\frac{\d_*}2 (t-r)}\Big(e^{-\frac {\d_*}4(T-r)}\sum_{k=1}^2\dbE|\bar X_{\sc k,T}(r)|^2+\dbE|v_{\sc k,T}(r)-v_{\sc k,\i}(r)|^2\Big)dr\\
&\leq K\int_0^te^{-\frac{\d_*}2 (t-r)}e^{-\frac {\d_*}4(T-r)} e^{-\frac{\d_*}2 (r-s)}|x|^2dr+e^{-{\d_*\over8}(T-t)}\int_0^\i e^{-{\d_*\over4}|t-r|}\xi(r)dr\\
&\leq K
e^{-\frac{\d_*} 4(T-t)}\(e^{-\frac{\d_*}4(t-s)}|x|^2+\int_0^\i e^{-{\d_*\over4}|t-r|}\xi(r)dr\).
\end{align*}
Using \eqref{v-v} again, we have
\begin{align*}
&\sum_{k=1}^2\dbE\int_0^te^{-\frac{\d_*}4(t-r)}|\bar u_{\sc k,T}(r)-\bar u_{\sc k,\i}(r))|^2dr\\
&\leq \sum_{k=1}^2\dbE\int_0^te^{-\frac{\d_*}4(t-r)}|\Th_{\sc k,T}(\a(r))|^2|\bar X_{\sc k,T}(r)-\bar X_{\sc k,\i}(r))|^2dr\\
&\q+\sum_{k=1}^2\dbE\int_0^te^{-\frac{\d_*}4(t-r)}|\Th_{\sc k,T}(\a(r))-\Th_{\sc k,\i}(\a(r))|^2|\bar X_{\sc k,\i}(r))|^2dr\\
&\q+K\sum_{k=1}^2\dbE\int_0^te^{-\frac{\d_*}4(t-r)}|v_{\sc k,T}(r,\a(r))-v_{\sc k,\i}(r,\a(r))|^2dr\\
&\leq Ke^{-\frac{\d_*} 4(T-t)}e^{-\frac{\d_*}4t}|x|^2+Ke^{-\frac{\d_*} 2(T-t)}e^{-\frac{\d_*}2t}|x|^2+Ke^{-{\d_*\over8}(T-t)}\int_0^\i e^{-{\d_*\over4}|t-s|}\xi(s)ds\\
&\leq Ke^{-\frac {\d_*} 8(T-t)}\(e^{-\frac{\d_*}4t}|x|^2+\int_0^\i e^{-{\d_*\over4}|t-r|}\xi(r)dr\).
\end{align*}
The proof is complete.

\end{proof}

Until now, we have proved the strong Turnpike property for Problem (MF-LQ)$_{\scT}$ as $T\rightarrow\i.$ One can see that the key of the process lies in deriving the control strategy \eqref{optintcont} by studying the convergence of Riccati equations in \eqref{ARE00} and BSDEs in \eqref{BSDE-Y10}. With appropriate assumptions, one can conclude those two systems converge to \eqref{ARE00int} and \eqref{BSDE-INT} respectively. 
Now the rest of this section aims  to  examine the optimality  $\bar u_{\sc1,\i}(\cd)\oplus\bar u_{\sc2,\i}(\cd)$. We will see that $\bar u_{\sc1,\i}(\cd)\oplus\bar u_{\sc2,\i}(\cd)$ is the optimal control  for either Problem (MF-LQ)$_{\sc\i}^*$ or Problem (MF-LQ)$_{\sc E}^*$ under different assumptions. We have two different cases.
\ms

\noindent{ \it Integrable Case.} Instead of (A3), we assume the following: \ms

{\bf (IC).} $b(\cd),\si(\cd),q(\cd),\bar q(\cd)\in L_{\dbF}^2(0,\i;\dbR^n),\q r(\cd),\bar r(\cd)\in L_{\dbF}^2(0,\i;\dbR^m).$
\ms

It is obvious that (IC) is stronger than (A3). Therefore, all the previous results hold in such a case. Recall the Problem (MF-LQ)$_{\sc\i}$ (or Problem (MF-LQ)$^*_{\sc\i}$ equivalently). By \cite{Mei-Wei-Yong-2025}, we directly have the following proposition.

\begin{proposition}Under (A1), (A2) and (LC),
$(\bar X_{\sc\i}^{0,x,\imath}(\cd),\bar u_{\sc\i}^{0,x,\imath}(\cd))$ is the unique optimal pair for Problem (MF-LQ)$_{\sc\i}$.
\end{proposition}

In this case, it follows that  $\bar X_{\sc\i}^{0,x,\imath}(\cd)\in L^2_{\dbF}(0,\i;\dbR^n)$ and therefore we call such a case by integrable case.\ms

\noindent{\it Non-Integrable Case.}
In addition to (A3), we further assume the following\ms 

\noindent {\bf (LIC).} 
\begin{equation}\label{boundzeta}\limsup_{T\rightarrow\infty}\frac1T\int_0^T\xi(t)dt<\infty.\end{equation}

In this case, we can verify that $\bar u_{\sc\i}^{0,x,\imath}(\cd)=\bar u_{\sc1,\i}^{0,x,\imath}(\cd)\oplus\bar u_{\sc2,\i}^{0,x,\imath}(\cd)$ is the optimal control of Problem (MF-LQ)$_{\sc E}^*$ as follows.

\begin{proposition} Suppose (A1), (A2), (A3) and (LIC) hold. For any $(0,x,\imath)\in\cD$, $\bar    u^{0,x,\imath}_{\sc\infty}(\cd)$ is the optimal control and $\bar X^{0,x,\imath}_{\sc\infty}(\cd)$ is the corresponding optimal trajectory   for Problem (MF-LQ)$_{\sc E}$.
Moreover, $J_{\sc E}(0,x,\imath;\bar u^{0,x,\imath}_{\sc\infty}(\cd))$ is finite.
\end{proposition}

\begin{proof} Without loss generality, we assume (A2)' instead of (A2). We also suppress the top index $(0,x,\imath)$ in the proof.
From \eqref{festimate1}--\eqref{boundEXTX}, it follows that
\begin{align}\label{bdinT}&\limsup_{T\rightarrow\infty}\frac1T\dbE\int_0^T|\bar u_{\sc1,T}(t)|^2+|\bar u_{\sc1,\i}(t)|^2dt<\i\text{ and }\lim_{T\rightarrow\infty}\frac1T\dbE\int_0^T|\bar u_{\sc1,T}(t)-\bar u_{\sc1,\i}(t)|^2dt=0.\end{align}

 Next, we see
\begin{align*}
&J_{\sc E}(0,x_1\oplus x_2,\imath;u_1(\cd)\oplus u_2(\cd))\ges J_{\sc E}(0,x_1\oplus x_2,\imath;\bar u_{\sc 1,\i}(\cd)\oplus \bar u_{\sc 2,\i})\\
&\q-\liminf_{T\rightarrow\infty}\frac KT \sum_{k=1}^2\int_0^ T\Big(\dbE[|\bar X_{\sc k,T}(t)-\bar X_{\sc k,\i}(t)|^2+|\bar u_{\sc k,T}(t)-\bar u_{\sc k,\i}(t)|^2]\\
&\qq\qq\qq\cdot\dbE[1+|\bar X_{\sc k,T}(t)|^2+|\bar X_{\sc k,\i}(t)|^2+|\bar u_{\sc k,T}(t)|^2+|\bar u^{x,\imath}_{\ k,\i}(t)|^2]\Big)^{\frac12}dt\\
&\q-\frac KT\liminf_{T\rightarrow\infty} \sum_{k=1}^2\int_0^ T\Big(\dbE[|\bar X_{\sc k,T}(t)-\bar X_{\sc k,\i}(t)|^2+|\bar u_{\sc k,T}(t)-\bar u_{\sc k,\i}(t)|^2]\Big)^{\frac12}\\
&\qq\qq\qq\cdot\(\dbE[1+|\bar X_{\sc k,T}(t)|^2+|\bar X_{\sc k,\i}(t)|^2+| \bar u_{\sc k,\infty}(t)|^2+| \bar u_{\sc k,T}(t)|^2]\)^{\frac12}dt.
\end{align*}
Taking $T\rightarrow\infty$, it follows that for any $u(\cd)\in\cU_{loc}[0,\infty)$,
\begin{align*}&J_{\sc E}(0,x_1\oplus x_2,\imath;u_1(\cd)\oplus u_2(\cd))\ges\liminf_{T\rightarrow\infty}\frac1T J_{\sc E}(0,x_1\oplus x_2,\imath;\bar u_{\sc1,T}(\cd)\oplus \bar u_{\sc2,T}(\cd))\\
&=J_{\sc E}(0,x_1\oplus x_2,\imath;\bar u_{\sc1,\i}(\cd)\oplus \bar u_{\sc2,\i}(\cd)). \end{align*}
Moreover, the uniform boundedness of $\dbE|\bar X_{\sc k,\i}(\cd)|^2$ and \eqref{bdinT} together imply that $J_{\sc E}(0,x_1\oplus x_2,\imath;\bar u_{\sc\i}(\cd)\oplus \bar u_{\sc2,\i}(\cd))$ is finite. Moreover, $\bar u_{\sc1,\i}(\cd)\oplus \bar u_{\sc2,\i}(\cd)$ is the optimal control process in $\cU_{loc}[0,\i)$ and $\bar X_{\sc 1,\i}(\cd)\oplus \bar X_{\sc 2,\i}(\cd)$ is the corresponding  trajectory for Problem (MF-LQ)$_{\sc E}$.

\end{proof}

In this case, it follows that   $\bar X_{\sc\i}^{0,x,\imath}(\cd)\in L^{2,loc}_{\dbF}(0,\i;\dbR^n)$ and therefore we call such a case by local-integrable case.

\section{Concluding Remarks}\label{sec:con}

In this paper, we obtained the turnpike property for mean-field LQ optimal control in an infinite horizon with a regime-switching state. To work with the mean-field terms, an orthogonal decomposition method is introduced.  Based on the integrability of the non-homogeneous terms over the infinite horizon, we prove that the limit process verifies two different types of optimalities: integrable cases and local-integrable case.  The idea in the paper is applicable  in future works on the strong turnpike property for the equilibrium strategies for LQ two-player games with mean-field interactions. We hope to report those results in future works.

\ms

\section{Declaration} The authors claim there is no conflict of interests. The authors did not use any  generative AI or AI-assisted technologies in the writing process of this paper.

\end{document}